\documentclass{amsart}
\usepackage[inner=3cm, outer=3cm, bottom=3cm]{geometry}
\usepackage{graphicx}

\usepackage{amssymb, amsmath, amsfonts}
\usepackage{url}
\usepackage[pdfborder={0 0 0}]{hyperref}
\usepackage{verbatim} 
\usepackage[all]{xy}
\usepackage{color}

\usepackage{marginnote}
\allowdisplaybreaks
\usepackage{enumitem}
\makeatletter
\newcommand{\mylabel}[2]{#2\def\@currentlabel{#2}\label{#1}}

\newcommand{\Linf}{{L^\infty}}

\allowdisplaybreaks
\newtheorem{theorem}{Theorem}[section]

\newtheorem{lemma}[theorem]{Lemma}

\newtheorem{example}[theorem]{Example}
\newtheorem{remark}[theorem]{Remark}

\numberwithin{equation}{section}

\newif\ifcolor
\colortrue  

\newenvironment{aalign}
{ 
	\csname align*\endcsname
}
{ 
	\csname endalign*\endcsname
}

\usepackage{etoolbox}
\newcommand\numberthis{\addtocounter{equation}{1}\tag{\theequation}}
\expandafter\appto\csname aalign\endcsname{\numberthis}
\allowdisplaybreaks

\begin{document}

\raggedbottom

\title{Traveling Waves for the Nonlinear Variational Wave Equation}

\author[K. Grunert]{Katrin Grunert}
\address{Department of Mathematical Sciences\\ NTNU Norwegian University of Science and Technology\\ NO-7491 Trondheim\\ Norway}
\email{katrin.grunert@ntnu.no}
\urladdr{\url{http://www.ntnu.no/ansatte/katrin.grunert}}

\author[A. Reigstad]{Audun Reigstad}
\address{Department of Mathematical Sciences\\ NTNU Norwegian University of Science and Technology\\ NO-7491 Trondheim\\ Norway}
\email{audun.reigstad@uit.no}
\urladdr{\url{http://www.ntnu.no/ansatte/audun.reigstad}}

\thanks{Research supported by the grants {\it Waves and Nonlinear Phenomena (WaNP)} and {\it Wave Phenomena and Stability --- a Shocking Combination (WaPheS)} from the Research Council of Norway.}  

\begin{abstract} 
We study traveling wave solutions of the nonlinear variational wave equation. In particular, we show how to obtain global, bounded, weak traveling wave solutions from local, classical ones. The resulting waves consist of monotone and constant segments, glued together at points where at least one one-sided derivative is unbounded.

Applying the method of proof to the Camassa--Holm equation, we recover some well-known results on its traveling wave solutions.
\end{abstract}

\maketitle

\section{Introduction}

We consider the nonlinear variational wave (NVW) equation
\begin{equation}
\label{eq:nvw}
	u_{tt}-c(u)(c(u)u_{x})_{x}=0,
\end{equation}
with initial data
\begin{equation}
\label{eq:initdat}
	u|_{t=0}=u_0 \quad \text{and} \quad u_t|_{t=0}=u_1.
\end{equation}
Here, $u=u(t,x)$ where $t\geq 0$ and $x\in\mathbb{R}$. 

The NVW equation was introduced by Saxton in \cite{Saxt:89}, where it is derived by applying the variational principle to the functional
\begin{equation*}
\int_{0}^{\infty}\int_{-\infty}^{\infty}(u_{t}^{2}-c^{2}(u)u_{x}^{2})\,dx\,dt.
\end{equation*}
The equation appears in the study of liquid crystals, where it describes the director field of a nematic liquid crystal, and where the function $c$ is given by
\begin{equation}
\label{eq:c}
	c^{2}(u)=\lambda_{1}\sin^{2}(u)+\lambda_{2}\cos^{2}(u),
\end{equation}
where $\lambda_{1}$ and $\lambda_{2}$ are positive physical constants. We refer to \cite{HS:91} and \cite{Saxt:89} for information about liquid crystals, and the derivation of the equation. 

It is well known that derivatives of solutions of the NVW equation can develop singularities in finite time even for smooth initial data, see \cite{GlaHunZh:96}. A singularity means that either $u_{x}$ or $u_{t}$ becomes unbounded pointwise while $u$ remains continuous. The continuation past singularities is highly nontrivial, and allows for various distinct solutions. The most common way of continuing the solution is to require that the energy is non-increasing, which naturally leads to the following two notions of solutions: Dissipative solutions for which the energy is decreasing in time, see \cite{BreHua16,ZhaZhen:03,ZhaZhen:05a,ZhaZhen:05}, and conservative solutions for which the energy is constant in time. In the latter case a semigroup of solutions has been constructed in \cite{BreZhe:06,HolRay:11}. 

We are interested in traveling wave solutions of \eqref{eq:nvw} with wave speed $s\in\mathbb{R}$, i.e., solutions of the form $u(t,x)=w(x-st)$ for some bounded and continuous function $w$. 

A bounded traveling wave was constructed in \cite{GlaHunZh:96}, corresponding to the function $c$ given in \eqref{eq:c}. The constructed wave is a weak solution, which is continuous and piecewise smooth. In particular, the smooth parts are monotone and at their endpoints cusp singularities might turn up. By the latter we mean points where the derivative is unbounded while the solution itself is bounded and continuous.

In this paper we consider local, classical traveling wave solutions of \eqref{eq:nvw}, i.e., solutions of the form $u(t,x)=w(x-st)$, where $w\in  C^2(I)$ for some interval $I$ and solves \eqref{eq:psiDer}, and study whether these can be glued together to produce globally bounded traveling waves. The approach we use is similar to the derivation of the Rankine--Hugoniot condition for hyperbolic conservation laws, see e.g. \cite{HolRis} and hence requires a minimal positive distance between any two gluing points.

We assume that the function $c$ belongs to $C^{2}(\mathbb{R})$ and that there exists $0<\alpha<\beta<\infty$, such that  
\begin{equation}
\label{eq:cass}
\alpha=\min_{u\in\mathbb{R}}c(u) \quad \text{and} \quad \beta=\max_{u\in\mathbb{R}}c(u).
\end{equation}
Moreover, we assume that
\begin{equation}
\label{eq:cderassumption}
\max_{u\in \mathbb{R}}|c'(u)|\leq K_{1} \quad \text{and} \quad \max_{u\in \mathbb{R}}|c''(u)|\leq K_{2}
\end{equation}
for positive constants $K_{1}$ and $K_{2}$.

The following theorem is our main result, and will be proved in the next section.

\begin{theorem}
\label{thm:nvw}
Let $c\in C^{2}(\mathbb{R})$ such that $\alpha$ and $\beta$ defined in \eqref{eq:cass} satisfy $0<\alpha<\beta<\infty$. Consider the continuous function $w:\mathbb{R}\mapsto\mathbb{R}$ composed of local, classical traveling wave solutions of \eqref{eq:nvw} with wave speed $s\in\mathbb{R}$. If $w$ is a global traveling wave to \eqref{eq:nvw}, then the following holds: 

If $|s|\notin[\alpha,\beta]$, then $w$ is a monotone, classical solution, which is globally unbounded. 

If $|s|\in(\alpha,\beta)$, two local, classical traveling wave solutions can only be glued together at points $\xi$ such that $\vert s\vert = c(w(\xi))$ and we have the following three possibilities:

1. If for some $\xi$, $|s|\neq c(w(\xi))$ and $c$ has a local maximum or minimum at $w(\xi)$, then the wave $w$ is a monotone, classical solution near $\xi$, which has an inflection point at $\xi$.  

2. If for some $\xi$, $|s|=c(w(\xi))$ and $c'(w(\xi))\neq 0$, then the wave $w$ is either constant or has a singularity at $\xi$, meaning that the derivative is unbounded at $\xi$ while $w$ is continuous. Near the singularity, the wave is a monotone, classical solution on both sides of $\xi$. The following scenarios are possible:

i) The derivative has the same sign (nonzero) on both sides of $\xi$, and the wave has an inflection point at $\xi$.

ii) The derivative has opposite sign (nonzero) on each side of $\xi$. Then, the wave is either convex or concave on both sides, and the singularity is a cusp. 

iii) The wave can be constant on one side of the singularity and strictly monotone on the other side.

3. If for some $\xi$, $|s|=c(w(\xi))$ and $c'(w(\xi))=0$, then the wave $w$ is constant.

For $|s|\in[\alpha,\beta]$, a weak bounded traveling wave solution of \eqref{eq:nvw} can be constructed.
\end{theorem}

\begin{figure}
	\centerline{\hbox{\includegraphics[width=8cm]{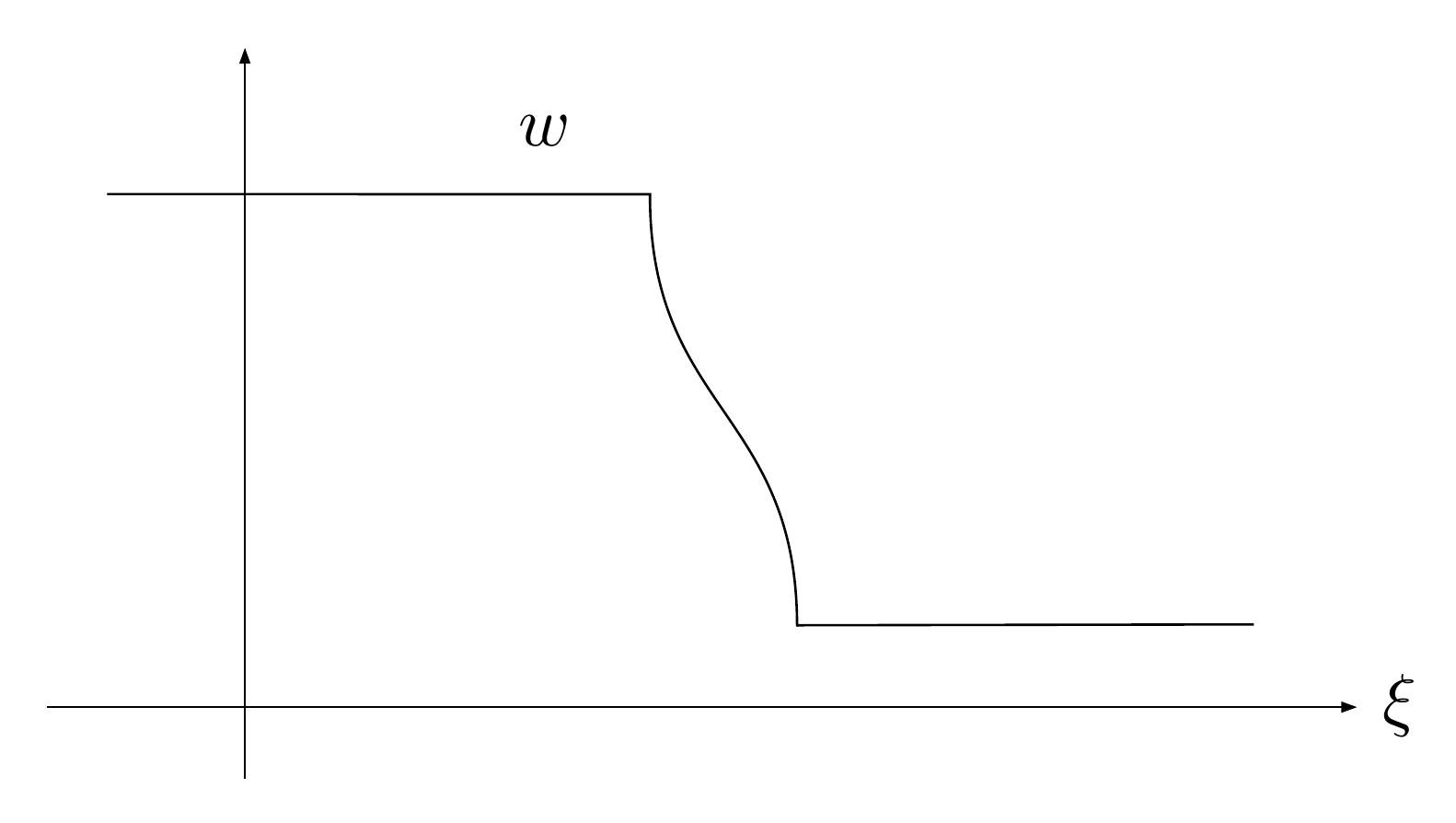}}}
	\caption{A traveling wave solution $w(\xi)$ consisting of two constant values joined together by a strictly decreasing part.}
	\label{fig:FigSW.pdf}
\end{figure}

We observe that Case 2 of Theorem \ref{thm:nvw} allows for globally bounded waves $w$. Excluding the trivial case of $w$ constant on the whole real line, we then see that the wave consists of increasing, constant, and decreasing parts, and that it has at least two singularities. The simplest nontrivial traveling wave consists of two constant values joined together by a monotone segment, which has two singularities, see Figure \ref{fig:FigSW.pdf}. This requires that $c(u)$ satisfies $c^2(u)=s^2$ for at least two values of $u$ and is illustrated in the following example.

Let $s\in \mathbb{R}$ such that $s^2>1$. Consider the periodic function 
\begin{equation}
c(u)=\sqrt{\sin(u)+s^2},
\end{equation}
which belongs to $C^2(\mathbb{R})$ and satisfies 
\begin{equation*}
0<\sqrt{s^2-1}\leq c(u)\leq \sqrt{s^2+1} \quad \text{ for all }u\in \mathbb{R}.
\end{equation*}
According to Theorem~\ref{thm:nvw}, possible gluing points can be identified by finding all points such that $c^2(u)=s^2$, which holds if and only if $u=n\pi$ ($n\in \mathbb{Z}$). Direct calculations yield that 
\begin{equation*}
c'(u)=\frac{\cos(u)}{2c(u)}=\pm\frac1{2\vert s\vert}\not=0 \quad \text{ for all } \quad u=n\pi \quad (n\in \mathbb{Z}).
\end{equation*}
Thus Theorem~\ref{thm:nvw}, Case 2 states that it is possible to construct a global traveling wave solution $u(t,x)=w(x-st)$ of the form 
\begin{equation}
w(\xi)=\begin{cases} \pi, & \xi\leq -\bar\xi,\\
                                 \tilde w(\xi), & -\bar \xi\leq \xi\leq \bar\xi,\\
                                 0, & \bar \xi<\xi,
           \end{cases} 
\end{equation}
if there exists a local, classical traveling wave solution $\tilde w$ connecting $0$ and $\pi$. That such a function $\tilde w$ exists will be shown next. Let $k\in \mathbb{R}\backslash \{0\}$ and assume that $\tilde w(0)=\frac{\pi}2$, then, cf. \eqref{eq:psiDer}, $\tilde w$ must be a local solution to 
\begin{equation}\label{ex:tw}
\tilde w_\xi(\xi)=-\frac{\sqrt{\vert k\vert}}{\sqrt{\vert s^2-c^2(\tilde w(\xi))\vert}}=-\frac{\sqrt{\vert k\vert}}{\sqrt{\sin(\tilde w(\xi))}}.
\end{equation}
Furthermore, the differential equation implies that $\tilde w_\xi(\xi)=\tilde w_\xi(-\xi)$ for all $\xi$ and hence $\tilde w(\xi)$ has an inflection point at $\xi=0$. This is as expected by Theorem~\ref{thm:nvw}, Case 1, since $c(u)$ attains a local maximum at $u=\frac{\pi}{2}=\tilde w(0)$. Moreover, one has 
\begin{equation}\label{tw:sy}
\tilde w(\xi)-\tilde w(0)=-(\tilde w(-\xi)-\tilde w(0)).
\end{equation}
Instead of computing the exact solution to \eqref{ex:tw}, we show that there exists $\bar \xi>0$ dependent on $\vert k\vert $ such that 
\begin{equation*}
\tilde w(-\bar \xi)=\pi \quad \text{ and }\quad \tilde w(\bar \xi)=0.
\end{equation*}
Due to \eqref{tw:sy} it suffices to show that there exists $\bar \xi>0$ such that $\tilde w(\bar\xi)=0$. From \eqref{ex:tw}, we get
\begin{equation}\label{tw:ex2}
-\int_0^{\frac{\pi}{2}} \sqrt{\sin(x)}dx=\int_0^{\bar \xi} \sqrt{\sin(\tilde w(l))}w_\xi(l)dl=-\sqrt{\vert k\vert}\bar \xi<0.
\end{equation}
Since $\sin(x)$ is positive on $[0,\pi]$, one has 
\begin{equation*}
\sin(x)\leq \sqrt{\sin(x)}\leq 1 \quad \text{ for all } x\in [0,\pi],
\end{equation*}
which implies
\begin{equation}\label{tw:ex3}
-\frac{\pi}{2}=-\int_0^{\frac{\pi}{2}} dx\leq -\int_0^{\frac{\pi}{2}}\sqrt{\sin(x)}dx\leq-\int_0^{\frac{\pi}{2}} \sin(x)dx=-1.
\end{equation}
Combining \eqref{tw:ex2} and \eqref{tw:ex3}, one ends up with
\begin{equation*}
\frac{1}{\sqrt{\vert k \vert}}\leq \bar \xi\leq \frac{\pi}{2\sqrt{\vert k\vert}},
\end{equation*}
which proves the existence of a local, classical traveling wave $\tilde w$ connection $0$ and $\pi$.

\vspace{0.2cm}
In Section~\ref{sec:CH}, we consider the Camassa--Holm (CH) equation
\begin{equation}
\label{eq:CH}
	u_{t}-u_{txx}+3uu_{x}-2u_{x}u_{xx}-uu_{xxx}=0,
\end{equation}
which was introduced in \cite{CamHol}. The CH equation has been studied intensively within the last three decades. There are too many interesting results to mention here, and we refer to \cite{BreCons:07-1,BreCons:07-2,CamHol,ConEs98-1,ConEs98-2,GruHolRay:14,HolRay:07, HolRay:09} and the citations therein for more information. We point out that the peakon solution, which was already observed in \cite{CamHol}, is a weak traveling wave solution of \eqref{eq:CH}. This is in contrast to the NVW equation, where there are no known non-constant explicit weak solutions. Moreover, like the NVW equation, singularity formation in the derivatives of solutions to \eqref{eq:CH} may occur, see \cite{ConEs98-2}.

In \cite{Len}, Lenells derives criteria for gluing together local, classical traveling wave solutions of \eqref{eq:CH} to obtain global, bounded traveling waves, see also \cite{Len2}. By doing so, all weak, bounded traveling wave solutions of the CH equation are classified. Some of these traveling waves have discontinuous derivatives, such as peakons, cuspons, stumpons, and composite waves. These waves have, except for the peakons, singularities in their derivatives. 

We apply the aforementioned method to the CH equation and reproduce the criteria derived by Lenells.

\section{Proof of Theorem~\ref{thm:nvw}}

Let $\xi=x-st$ and denote the derivative of $w$ with respect to $\xi$ by $w_{\xi}$. Assume for the moment that $w\in C^{2}(\mathbb{R})$. Inserting the derivatives of $w$ into \eqref{eq:nvw} yields
\begin{equation}
\label{eq:TWnvw}
	\big[s^{2}-c^{2}(w)\big]w_{\xi\xi}-c(w)c'(w)w_{\xi}^{2}=0.
\end{equation}
Observe that \eqref{eq:TWnvw} is satisfied at all points $\xi$ such that $|s|=c(w(\xi))$, at which either $w_{\xi}(\xi)=0$, leading to constant solutions, or $c'(w(\xi))=0$. We multiply \eqref{eq:TWnvw} by $2w_{\xi}$ and get
\begin{equation*}
	\frac{d}{d\xi}\Big(w_{\xi}^{2}\big[s^{2}-c^{2}(w)\big]\Big)=0.
\end{equation*}
Integration leads to
\begin{equation}
\label{eq:NVWClassicalWave}
	w_{\xi}^{2}\big[s^{2}-c^{2}(w)\big]=k
\end{equation}
for some integration constant $k$. Observe that we derived \eqref{eq:NVWClassicalWave} assuming that $w\in C^{2}(\mathbb{R})$, but for \eqref{eq:NVWClassicalWave} to make sense it suffices that $w$ is in $C^{1}(\mathbb{R})$. 

We say that $u$ is a local, classical traveling wave solution of \eqref{eq:nvw} if $u(t,x)=w(x-st)$, for some $w$ in $C^{2}(I)$, where $I$ denotes some interval, and satisfies \eqref{eq:TWnvw}. 

If $|s|\notin[\alpha,\beta]$, then $|s|\neq c(w(\xi))$ for all $\xi$ and we have
\begin{equation}
\label{eq:psiDer}
	w_{\xi}(\xi)=\pm\frac{\sqrt{|k|}}{\sqrt{|s^{2}-c^{2}(w(\xi))|}}.
\end{equation}
The right-hand side of \eqref{eq:psiDer} is Lipschitz continuous with respect to $w$ and there exists a unique local solution $w$ which is continuously differentiable and monotone. For these solutions we see from \eqref{eq:psiDer} that the derivatives are bounded. In particular, the solutions are bounded locally, but not globally.

In the case $|s|\in[\alpha,\beta]$, Lipschitz continuity fails, and the standard existence and uniqueness result for ordinary differential equations does not apply. In this case we show, under some specific conditions, that if there is a local solution, it is H\"older continuous. Let $w$ be a bounded and strictly monotone solution of \eqref{eq:psiDer} on an interval $[\xi_{0},\xi_{1}]$ such that  $c(w(\xi_{0}))\neq |s|$, $c(w(\xi_{1}))\neq |s|$, and $|s|=c(w(\eta))$ for some $\eta\in(\xi_{0},\xi_{1})$. Then, by assumption, the derivative $w_{\xi}$ is bounded at $\xi_{0}$ and $\xi_{1}$. We claim that the solution is H\"older continuous on $[\xi_{0},\xi_{1}]$ if $c'(w(\eta))\neq 0$. From \eqref{eq:psiDer} and a change of variables we get
\begin{equation}
\label{eq:waveL2norm}
	\int_{\xi_{0}}^{\xi_{1}}w_{\xi}^{2}(\xi)\,d\xi=\sqrt{|k|}\bigg|\int_{w(\xi_{0})}^{w(\xi_{1})}\frac{1}{\sqrt{|s^{2}-c^{2}(z)|}}\,dz\bigg|,
\end{equation}
and $c(w(\eta))=\vert s\vert$ yields
\begin{equation*}
	\int_{\xi_{0}}^{\xi_{1}}w_{\xi}^{2}(\xi)\,d\xi=\sqrt{|k|}\bigg|\int_{w(\xi_{0})}^{w(\xi_{1})}\frac{1}{\sqrt{|c^{2}(w(\eta))-c^{2}(z)|}}\,dz\bigg|.
\end{equation*}
The integrand is finite everywhere except at $z=w(\eta)$. For $z$ near $w(\eta)$ we replace $c(z)$ by its Taylor approximation and get
\begin{align*}
	|c^{2}(w(\eta))-c^{2}(z)|&=|c(w(\eta))+c(z)|\cdot|c(w(\eta))-c(z)|\\
	&\geq 2\alpha|c(w(\eta))-c(z)|=2\alpha\big|c'(w(\eta))(z-w(\eta))+\frac{1}{2}c''(p)(z-w(\eta))^{2}\big|,
\end{align*}
for some $p$ between $z$ and $w(\eta)$. The expression
\begin{equation*}
	\big|c'(w(\eta))(z-w(\eta))+\frac{1}{2}c''(p)(z-w(\eta))^{2}\big|^{-\frac{1}{2}}
\end{equation*}
is integrable if $c'(w(\eta))\neq 0$ and not integrable if $c'(w(\eta))=0$. Therefore, the integral $\int_{\xi_{0}}^{\xi_{1}}w_{\xi}^{2}(\xi)\,d\xi$
is finite if for all $\xi\in(\xi_{0},\xi_{1})$ such that $|s|=c(w(\xi))$, we have $c'(w(\xi))\neq 0$. In particular, by the Cauchy--Schwarz inequality we have
\begin{equation*}
	|w(\xi_{1})-w(\xi_{0})|\leq ||w_{\xi}||_{L^{2}([\xi_{0},\xi_{1}])}|\xi_{1}-\xi_{0}|^{\frac{1}{2}}
\end{equation*}
and $w$ is H\"older continuous with exponent $\frac{1}{2}$. This continuity will be important later in the text when we discuss which traveling waves can be glued together.

We illustrate the above result with an example. 
\begin{example}
Let $A=\frac{\beta-\alpha}{\pi}$ and $B=\alpha+\beta$, where $0<\alpha<\beta<\infty$. Consider the function
\begin{equation}\label{specificc}
	c(u)=A\arctan(u)+\frac{B}{2},
\end{equation}
which is strictly increasing and satisfies $\displaystyle\lim_{u\rightarrow-\infty}c(u)=\alpha$ and  $\displaystyle\lim_{u\rightarrow+\infty}c(u)=\beta$. Consider the wave speed $s=\frac{B}{2}$, where we have $\alpha<s<\beta$. 
Let $f(u)=s^{2}-c^{2}(u)$. We have $f(u)=-A\arctan(u)(A\arctan(u)+B)$. We compute the derivative and get $f'(u)=-\frac{A}{1+u^{2}}(2A\arctan(u)+B)$,
and since $2A\arctan(u)+B\geq 2\alpha>0$ for all $u$ we have $f'(u)<0$. The only point satisfying $f(u)=0$ is $u=0$. In other words, $s=c(0)$. 

Denote by $w$ the strictly increasing solution to \eqref{eq:psiDer} and \eqref{specificc}. We assume that $w(\xi_{0})<0<w(\xi_{1})$, so that $c(w(\xi_{0}))\neq s$ and $c(w(\xi_{1}))\neq s$, which implies that the derivative $w_{\xi}$ is bounded at $\xi_{0}$ and $\xi_{1}$. From \eqref{eq:waveL2norm} we get
\begin{aalign}
\label{eq:excInt}
	\int_{\xi_{0}}^{\xi_{1}}w_{\xi}^{2}(\xi)\,d\xi&=\sqrt{|k|}\int_{w(\xi_{0})}^{0}\frac{1}{\sqrt{-A\arctan(z)(A\arctan(z)+B)}}\,dz\\
	&\quad +\sqrt{|k|}\int_{0}^{w(\xi_{1})}\frac{1}{\sqrt{A\arctan(z)(A\arctan(z)+B)}}\,dz.
\end{aalign}
By a change of variables we have 
\begin{align*}
	&\int_{0}^{w(\xi_{1})}\frac{1}{\sqrt{A\arctan(z)(A\arctan(z)+B)}}\,dz \leq \int_{0}^{w(\xi_{1})}\frac{1}{\sqrt{AB\arctan(z)}}\,dz\\
	&\leq (1+w^{2}(\xi_{1})) \int_{0}^{w(\xi_{1})}\frac{1}{\sqrt{AB\arctan(z)}}\frac{1}{1+z^{2}}\,dz=2(1+w^{2}(\xi_{1}))\frac{1}{\sqrt{AB}}\sqrt{\arctan(w(\xi_{1}))}
\end{align*}
and since $w(\xi_{1})$ is finite, the integral converges. Note that this only holds locally. The first integral in \eqref{eq:excInt} can be treated in the same way, showing that $w_{\xi}\in L^{2}([\xi_{0},\xi_{1}])$ and we conclude that $w$ is H\"older continuous on $[\xi_{0},\xi_{1}]$.
\end{example}

Let us focus on weak traveling wave solutions. To derive the weak form of \eqref{eq:nvw} we first assume that we have a bounded solution $u\in C^{2}((0,\infty)\times\mathbb{R})$. We multiply \eqref{eq:nvw} by a smooth test function $\phi\in C_{c}^{\infty}((0,\infty)\times\mathbb{R})$ and integrate by parts, which yields
\begin{equation}
	\label{eq:weakForm}
	\int_{0}^{\infty}\int_{-\infty}^{\infty}\big[-u_{t}\phi_{t}+c(u)c'(u)u_{x}^{2}\phi+c^{2}(u)u_{x}\phi_{x}\big]\,dx\,dt=0.
\end{equation}
We say that a function $u$ satisfying $u(t,\cdot)\in\Linf(\mathbb{R})$ and $u_{t}(t,\cdot),u_{x}(t,\cdot)\in L^{2}(\mathbb{R})$ for all $t\geq 0$ is a weak solution of \eqref{eq:nvw} if \eqref{eq:weakForm} holds for all test functions $\phi$ in $C_{c}^{\infty}((0,\infty)\times\mathbb{R})$. We observe that if there exists a piecewise smooth traveling wave solution satisfying these conditions, it is H\"older continuous with exponent $\frac{1}{2}$. 

In the case of a traveling wave $u(t,x)=w(x-st)$, \eqref{eq:weakForm} reads
\begin{equation*}
	\int_{0}^{\infty}\int_{-\infty}^{\infty}\big[sw_{\xi}\phi_{t}+c(w)c'(w)w_{\xi}^{2}\phi+c^{2}(w)w_{\xi}\phi_{x}\big]\,dx\,dt=0.
\end{equation*}  

Now we want to glue together two local, classical solutions to produce a weak traveling wave solution. At the points where we glue them together the derivatives may not exist. Thus, we consider the following situation: assume that $u_{t}$ and $u_{x}$ have discontinuities that move along a smooth curve $\Gamma: x=\gamma(t)$, where we assume that $\gamma$ is a smooth and strictly increasing function. Moreover, we assume that there exists a sufficiently small neighborhood of $\gamma(t)$ such that $u$ is a classical solution of \eqref{eq:nvw} on each side of $\gamma(t)$.

\begin{lemma}
\label{lem:nvw}
Given a curve $\Gamma:x=\gamma(t)=st+\gamma_0$, where $\gamma_0$ is a constant, denote by $D$ a neighborhood of $(\bar t, \gamma(\bar t))\in \Gamma$. Furthermore, let $D=D_1\cup \Gamma\vert_D\cup D_2$, where $D_1$ and $D_2$ are the parts of $D$ to the left and to the right of $\Gamma$, respectively, see Figure \ref{fig:FigD}. Consider two local, classical traveling wave solutions $u_1$ and $u_2$  of \eqref{eq:nvw} in $D_1$ and $D_2$, respectively. Assume that we glue these waves at $\Gamma$ to obtain a continuous traveling wave $u(t,x)=w(x-st)$ in $D$, which satisfies 
\begin{equation*}
	\iint_{D}\big[sw_{\xi}\phi_{t}+c(w)c'(w)w_{\xi}^{2}\phi+c^{2}(w)w_{\xi}\phi_{x}\big]\,dx\,dt=0 \quad \text{ for any } \phi\in C_c^\infty(D).
\end{equation*}

If $|s|\notin[\alpha,\beta]$, then 
\begin{subequations}
\begin{equation}
\label{eq:Res1}
	w_{\xi}(\gamma_{0}-)=w_{\xi}(\gamma_{0}+).
\end{equation}

If $|s|\in(\alpha,\beta)$ and $c'(w(\xi))\neq 0$ for all $\xi=x-st$, such that $(t,x)\in D$ and $|s|=c(w(\xi))$, then 
\begin{aalign}
\label{eq:Res2}
	&\Big[\sqrt{|k_{1}|}\,\emph{sign}\big(\big((c^{2}(w(\gamma_{0}))-s^{2})w_{\xi}(\gamma_{0})\big)-\big)\\
	&\quad-\sqrt{|k_{2}|}\,\emph{sign}\big(\big((c^{2}(w(\gamma_{0}))-s^{2})w_{\xi}(\gamma_{0})\big)+\big)\Big]\sqrt{|c^{2}(w(\gamma_{0}))-s^{2}|}=0,
\end{aalign}
\end{subequations}
where $k_{1}$ and $k_{2}$ denote the constants in \eqref{eq:NVWClassicalWave} corresponding to the local, classical traveling wave solutions $u_1$ and $u_2$ in $D_{1}$ and $D_{2}$, respectively.
\end{lemma}

\begin{proof}
Let
\begin{equation*}
	I=\{t\in[0,\infty) \ | \ (t,\gamma(t))\in D\}.
\end{equation*}
\begin{figure}
	\centerline{\hbox{\includegraphics[width=8cm]{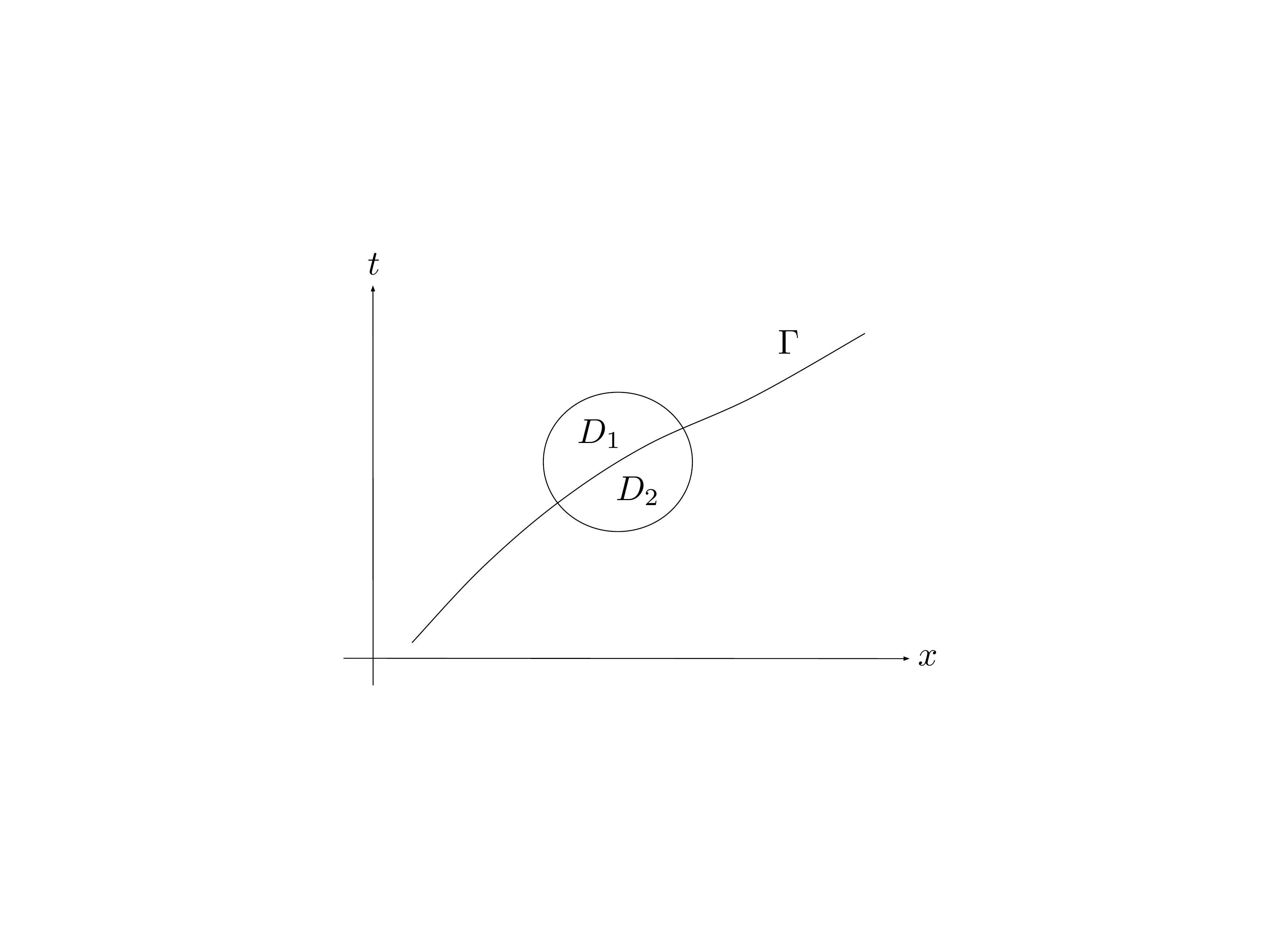}}}
	\caption{Some strictly increasing curve $\Gamma: x=\gamma(t)$ and the neighborhoods $D_{1}$ and $D_{2}$.}
	\label{fig:FigD}
\end{figure}
For any $\varepsilon>0$ consider
\begin{equation*}
	D_{i}^{\varepsilon}=\{(t,x)\in D_{i} \ | \ \text{dist}((t,x),\Gamma)>\varepsilon\}
\end{equation*}
for $i=1,2$. We have
\begin{aalign}
	\label{eq:weakFormD}
	&\iint_{D}\big[
	-u_{t}\phi_{t}+c(u)c'(u)u_{x}^{2}\phi+c^{2}(u)u_{x}\phi_{x}\big]\,dx\,dt\\
	&=\lim_{\varepsilon\rightarrow 0}\bigg(\sum_{i=1}^{2}\iint_{D_{i}^{\varepsilon}}\big[
	-u_{t}\phi_{t}+c(u)c'(u)u_{x}^{2}\phi+c^{2}(u)u_{x}\phi_{x}\big]\,dx\,dt\bigg). 
\end{aalign}
Since $u$ is a classical solution in $D_{1}^{\varepsilon}$, we have
\begin{align*}
	&\iint_{D_{1}^{\varepsilon}}\big[
	-u_{t}\phi_{t}+c(u)c'(u)u_{x}^{2}\phi+c^{2}(u)u_{x}\phi_{x}\big]\,dx\,dt=\iint_{D_{1}^{\varepsilon}}\big[(c^{2}(u)u_{x}\phi)_{x}-(u_{t}\phi)_{t}\big]\,dx\,dt.
\end{align*} 
By Green's theorem we get
\begin{equation*}
	\iint_{D_{1}^{\varepsilon}}\big[(c^{2}(u)u_{x}\phi)_{x}-(u_{t}\phi)_{t}\big]\,dx\,dt
	=\int_{\partial D_{1}^{\varepsilon}}\big[u_{t}\phi\,dx+c^{2}(u)u_{x}\phi\,dt\big]=\int_{\Gamma_{1}^{\varepsilon}}\big[u_{t}\phi\,dx+c^{2}(u)u_{x}\phi\,dt\big],
\end{equation*}
where the last equality follows since $\phi$ is zero everywhere on $\partial D_{1}^{\varepsilon}$ except on $\Gamma_{1}^{\varepsilon}$, where $\Gamma_{1}^{\varepsilon}$ is the part of $\partial D_{1}^{\varepsilon}$ which does not coincide with the boundary of $ D$. We can parametrize the curve by $\Gamma_{1}^{\varepsilon}: x=\gamma_{1}^{\varepsilon}(t)$ for $t\in I_{1}^{\varepsilon}$, where $\gamma_{1}^{\varepsilon}$ is a smooth and strictly increasing function and $I_{1}^{\varepsilon}$ is an interval. Now we have
\begin{aalign}
\label{eq:nvwD1epsilon}
	&\iint_{D_{1}^{\varepsilon}}\big[
	-u_{t}\phi_{t}+c(u)c'(u)u_{x}^{2}\phi+c^{2}(u)u_{x}\phi_{x}\big]\,dx\,dt\\
	&=\int_{I_{1}^{\varepsilon}}\big[(u_{t}\phi)(t,\gamma_{1}^{\varepsilon}(t))(\gamma_{1}^{\varepsilon})'(t)+(c^{2}(u)u_{x}\phi)(t,\gamma_{1}^{\varepsilon}(t))\big]\,dt.
\end{aalign}
By assumption $u(t,x)=w(x-st)$ is a classical traveling wave solution in $D_{1}^{\varepsilon}$. It follows that $\gamma_{1}^{\varepsilon}(t)=\gamma(t)-\varepsilon\sqrt{s^{2}+1}$,
\begin{equation*}
	I_{1}^{\varepsilon}=\{t\in[0,\infty) \ | \ (t,\gamma_{1}^{\varepsilon}(t))\in D_{1}\}, \quad \text{and} \quad
	\Gamma_{1}^{\varepsilon}=\{(t,\gamma_{1}^{\varepsilon}(t)) \ | \ t\in I_{1}^{\varepsilon}\}.
\end{equation*}
From \eqref{eq:nvwD1epsilon} we get
\begin{aalign}
\label{eq:nvwD1epsilon2}
	&\iint_{D_{1}^{\varepsilon}}\big[
	sw_{\xi}\phi_{t}+c(w)c'(w)w_{\xi}^{2}\phi+c^{2}(w)w_{\xi}\phi_{x}\big]\,dx\,dt\\
	&=\int_{I_{1}^{\varepsilon}}\big[c^{2}(w(\gamma_{1}^{\varepsilon}(t)-st))-s^{2}\big]w_{\xi}(\gamma_{1}^{\varepsilon}(t)-st)\phi(t,\gamma_{1}^{\varepsilon}(t))\,dt.
\end{aalign}

By similar computations, as above, for $D_2^\varepsilon$ we get
\begin{equation*}
	\iint_{D_{2}^{\varepsilon}}\big[
	-u_{t}\phi_{t}+c(u)c'(u)u_{x}^{2}\phi+c^{2}(u)u_{x}\phi_{x}\big]\,dx\,dt
	=\int_{\Gamma_{2}^{\varepsilon}}\big[u_{t}\phi\,dx+c^{2}(u)u_{x}\phi\,dt\big],
\end{equation*}
where $\Gamma_{2}^{\varepsilon}$ is the part of $\partial D_{2}^{\varepsilon}$ which does not coincide with the boundary $\partial D$. We parametrize the curve by $\Gamma_{2}^{\varepsilon}: x=\gamma_{2}^{\varepsilon}(t)$ for $t\in I_{2}^{\varepsilon}$, where $\gamma_{2}^{\varepsilon}$ is a smooth and strictly increasing function and $I_{2}^{\varepsilon}$ is an interval. We obtain
\begin{aalign}
\label{eq:nvwD2epsilon}
	&\iint_{D_{2}^{\varepsilon}}\big[
	-u_{t}\phi_{t}+c(u)c'(u)u_{x}^{2}\phi+c^{2}(u)u_{x}\phi_{x}\big]\,dx\,dt\\
	&=-\int_{I_{2}^{\varepsilon}}\big[(u_{t}\phi)(t,\gamma_{2}^{\varepsilon}(t))(\gamma_{2}^{\varepsilon})'(t)+(c^{2}(u)u_{x}\phi)(t,\gamma_{2}^{\varepsilon}(t))\big]\,dt
\end{aalign}
where the negative sign comes from the fact that we are integrating counterclockwise around the boundary in Green's theorem. Inserting $u(t,x)=w(x-st)$ yields
\begin{aalign}
\label{eq:nvwD2epsilon2}
	&\iint_{D_{2}^{\varepsilon}}\big[
	sw_{\xi}\phi_{t}+c(w)c'(w)w_{\xi}^{2}\phi+c^{2}(w)w_{\xi}\phi_{x}\big]\,dx\,dt\\
	&=-\int_{I_{2}^{\varepsilon}}\big[c^{2}(w(\gamma_{2}^{\varepsilon}(t)-st))-s^{2}\big]w_{\xi}(\gamma_{2}^{\varepsilon}(t)-st)\phi(t,\gamma_{2}^{\varepsilon}(t))\,dt,
\end{aalign}
where $\gamma_{2}^{\varepsilon}(t)=\gamma(t)+\varepsilon\sqrt{s^{2}+1}$,
\begin{equation*}
	I_{2}^{\varepsilon}=\{t\in[0,\infty) \ | \ (t,\gamma_{2}^{\varepsilon}(t))\in D_{2}\}, \quad \text{and} \quad
	\Gamma_{2}^{\varepsilon}=\{(t,\gamma_{2}^{\varepsilon}(t)) \ | \ t\in I_{2}^{\varepsilon}\}.
\end{equation*}

Consider $|s|\notin[\alpha,\beta]$. Then $|s|\neq c(w(\xi))$ for all $\xi$, and by \eqref{eq:NVWClassicalWave} the derivative $w_{\xi}$ is bounded at all points in $D$. From \eqref{eq:nvwD1epsilon2} and \eqref{eq:nvwD2epsilon2} we have
\begin{equation*}
	\lim_{\varepsilon\rightarrow 0}\iint_{D_{1}^{\varepsilon}}\big[
	sw_{\xi}\phi_{t}+c(w)c'(w)w_{\xi}^{2}\phi+c^{2}(w)w_{\xi}\phi_{x}\big]\,dx\,dt=\int_{I}\big[c^{2}(w(\gamma_{0}))-s^{2}\big]w_{\xi}(\gamma_{0}-)\phi(t,\gamma(t))\,dt
\end{equation*}
and 
\begin{equation*}
	\lim_{\varepsilon\rightarrow 0}\iint_{D_{2}^{\varepsilon}}\big[
	sw_{\xi}\phi_{t}+c(w)c'(w)w_{\xi}^{2}\phi+c^{2}(w)w_{\xi}\phi_{x}\big]\,dx\,dt=-\int_{I}\big[c^{2}(w(\gamma_{0}))-s^{2}\big]w_{\xi}(\gamma_{0}+)\phi(t,\gamma(t))\,dt,
\end{equation*}
respectively. Here, $w_{\xi}(\gamma_{0}-)$ and $w_{\xi}(\gamma_{0}+)$ denote the left and right limit of $w_{\xi}$ at $\gamma_{0}$, respectively. We insert these expressions in \eqref{eq:weakFormD}, and get
\begin{align*}
	&\iint_{D}\big[sw_{\xi}\phi_{t}+c(w)c'(w)w_{\xi}^{2}\phi+c^{2}(w)w_{\xi}\phi_{x}\big]\,dx\,dt\\
	&=\big[c^{2}(w(\gamma_{0}))-s^{2}\big]\big[w_{\xi}(\gamma_{0}-)-w_{\xi}(\gamma_{0}+)\big]\int_{I}\phi(t,\gamma(t))\,dt.
\end{align*}
For $w$ to be a weak solution this expression has to be zero for every test function $\phi$, and we must have
\begin{equation*}
	\big[c^{2}(w(\gamma_{0}))-s^{2}\big]\big[w_{\xi}(\gamma_{0}-)-w_{\xi}(\gamma_{0}+)\big]=0,
\end{equation*}
which implies $w_{\xi}(\gamma_{0}-)=w_{\xi}(\gamma_{0}+)$. This proves \eqref{eq:Res1}.

Now we consider $|s|\in[\alpha,\beta]$. In this case $w_{\xi}$ may be unbounded on the curve $\Gamma$, and we have to eliminate the derivatives from \eqref{eq:nvwD1epsilon2} and \eqref{eq:nvwD2epsilon2}. Recall that we only consider continuous waves. 

Since $w$ is a classical solution in $\overline{D_{1}^{\varepsilon}}$ we get from \eqref{eq:NVWClassicalWave},
\begin{equation}
\label{eq:D1eq}
	w_{\xi}^{2}(\xi)\big[s^{2}-c^{2}(w(\xi))\big]=k_{1}
\end{equation}
where $k_{1}$ is a constant. Thus, we have
\begin{equation*}
	(c^{2}(w)-s^{2})w_{\xi}=\text{sign}\big((c^{2}(w)-s^{2})w_{\xi}\big)\sqrt{|c^{2}(w)-s^{2}|}\sqrt{|k_{1}|}.
\end{equation*}
In \eqref{eq:nvwD1epsilon2} we now get
\begin{align*}
	&\iint_{D_{1}^{\varepsilon}}\big[
	sw_{\xi}\phi_{t}+c(w)c'(w)w_{\xi}^{2}\phi+c^{2}(w)w_{\xi}\phi_{x}\big]\,dx\,dt\\
	&=\int_{I_{1}^{\varepsilon}}\sqrt{|k_{1}|}\Big(\text{sign}\big((c^{2}(w)-s^{2})w_{\xi}\big)\sqrt{|c^{2}(w)-s^{2}|}\Big)(\gamma_{1}^{\varepsilon}(t)-st)\phi(t,\gamma_{1}^{\varepsilon}(t))\,dt,
\end{align*}
and we obtain
\begin{aalign}
\label{eq:limD1Epsilon}
	&\lim_{\varepsilon\rightarrow 0}\iint_{D_{1}^{\varepsilon}}\big[
	sw_{\xi}\phi_{t}+c(w)c'(w)w_{\xi}^{2}\phi+c^{2}(w)w_{\xi}\phi_{x}\big]\,dx\,dt\\
	&=\int_{I}\sqrt{|k_{1}|}\,\text{sign}\big(\big((c^{2}(w(\gamma_{0}))-s^{2})w_{\xi}(\gamma_{0})\big)-\big)\sqrt{|c^{2}(w(\gamma_{0}))-s^{2}|}\phi(t,\gamma(t))\,dt.
\end{aalign}

In a similar way, we get by using \eqref{eq:NVWClassicalWave} 
in $\overline{D_{2}^{\varepsilon}}$, 
\begin{aalign}
	\label{eq:limD2Epsilon}
	&\lim_{\varepsilon\rightarrow 0}\iint_{D_{2}^{\varepsilon}}\big[
	sw_{\xi}\phi_{t}+c(w)c'(w)w_{\xi}^{2}\phi+c^{2}(w)w_{\xi}\phi_{x}\big]\,dx\,dt\\
	&=-\int_{I}\sqrt{|k_{2}|}\,\text{sign}\big(\big((c^{2}(w(\gamma_{0}))-s^{2})w_{\xi}(\gamma_{0})\big)+\big)\sqrt{|c^{2}(w(\gamma_{0}))-s^{2}|}\phi(t,\gamma(t))\,dt,
\end{aalign}
for some constant $k_{2}$. Combining \eqref{eq:limD1Epsilon} and \eqref{eq:limD2Epsilon} in \eqref{eq:weakFormD}, we get
\begin{align*}
	&\iint_{D}\big[
	sw_{\xi}\phi_{t}+c(w)c'(w)w_{\xi}^{2}\phi+c^{2}(w)w_{\xi}\phi_{x}\big]\,dx\,dt\\
	&=\int_{I}\Big[
	\sqrt{|k_{1}|}\,\text{sign}\big(\big((c^{2}(w(\gamma_{0}))-s^{2})w_{\xi}(\gamma_{0})\big)-\big)\\
	&\hspace{30pt}-\sqrt{|k_{2}|}\,\text{sign}\big(\big((c^{2}(w(\gamma_{0}))-s^{2})w_{\xi}(\gamma_{0})\big)+\big)\Big]\sqrt{|c^{2}(w(\gamma_{0}))-s^{2}|}\phi(t,\gamma(t))\,dt.
\end{align*} 
For $w$ to be a weak solution this expression has to be zero for every test function, and we must have  
\begin{align*}
	&\Big[\sqrt{|k_{1}|}\,\text{sign}\big(\big((c^{2}(w(\gamma_{0}))-s^{2})w_{\xi}(\gamma_{0})\big)-\big)\\
	&\quad-\sqrt{|k_{2}|}\,\text{sign}\big(\big((c^{2}(w(\gamma_{0}))-s^{2})w_{\xi}(\gamma_{0})\big)+\big)\Big]\sqrt{|c^{2}(w(\gamma_{0}))-s^{2}|}=0.
\end{align*}
This concludes the proof of \eqref{eq:Res2}. 
\end{proof} 
 
\begin{remark} 
Note that \eqref{eq:nvwD1epsilon} and \eqref{eq:nvwD2epsilon} hold for any solution $u$ and curve $x=\gamma(t)$ as described before the lemma, not just for traveling wave solutions, where $\gamma(t)=st+\gamma_{0}$. 
\end{remark}

Using Lemma \ref{lem:nvw} we prove Theorem \ref{thm:nvw}.

\begin{proof}[Proof of Theorem \ref{thm:nvw}]

Consider $|s|\notin[\alpha,\beta]$. From \eqref{eq:Res1}, $w$ and its derivative $w_\xi$ are continuous at $\gamma_{0}$. In particular, $w$ is monotone and coincides with the global solution for \eqref{eq:psiDer} for a fixed $k$.  Thus, the resulting wave will be unbounded and hence not a weak solution to the NVW equation. We will not discuss this case further.

Now consider $|s|\in[\alpha,\beta]$. First we study the case $|s|\neq c(w(\gamma_{0}))$. For \eqref{eq:Res2} to be satisfied we must have
\begin{equation*}
	\sqrt{|k_{1}|}\,\text{sign}\big(w_{\xi}(\gamma_{0})-\big) -\sqrt{|k_{2}|}\,\text{sign}\big(w_{\xi}(\gamma_{0})+\big)=0.
\end{equation*}

If $\text{sign}\big(w_{\xi}(\gamma_{0})-\big)$ and $\text{sign}\big(w_{\xi}(\gamma_{0})+\big)$ have opposite sign we get $k_{1}=k_{2}=0$ and $w$ is constant in $D$.

If $\text{sign}\big(w_{\xi}(\gamma_{0})-\big)$ and $\text{sign}\big(w_{\xi}(\gamma_{0})+\big)$ have the same sign then $|k_{1}|=|k_{2}|$. Then the solution $w$ is monotone in $D$ and is given by \eqref{eq:psiDer}, where the constant $k$ is replaced by $k_1$. Since $w$ is a classical solution in $D_{1}$ and $D_{2}$, and $|s|\neq c(w(\gamma_{0}))$, we have $|s|\neq c(w(\xi))$ in $D$. Both $w$ and its derivative $w_\xi$ are continuous at $\gamma_{0}$. In particular, $w$ is monotone and coincides with the local solution in $D$ of the above differential equation for a fixed $k_1$. 

Thus, we showed that gluing solutions at points $\gamma_0$ so that $c(w(\gamma_0))\not =\vert s\vert$, does not yield a new solution. In particular, one can possibly only glue two solutions with different $k$ together at a point $\gamma_0$ to obtain a new solution, if $c(w(\gamma_0))=\vert s\vert$. This means in particular, for bounded, non-constant waves, that $c$ must have at least one extremal point and hence $w$ must have at least one inflection point by \eqref{eq:TWnvw}. 

Next we consider $|s|\in[\alpha,\beta]$ such that $|s|=c(w(\gamma_{0}))$. As discussed before, at points $\xi$ where $|s|=c(w(\xi))$ and $c'(w(\xi))=0$, $w_\xi(\xi)$ is unbounded and $w_\xi$ does not belong to $L^2_{\text{loc}}(\mathbb{R})$. Therefore, by the definition of a weak solution, we cannot use such waves as building blocks. This immediately excludes the cases $|s|=\alpha$ and $|s|=\beta$. Thus, we consider $|s|\in(\alpha,\beta)$ and assume that all points $\xi$ such that $|s|=c(w(\xi))$ satisfy $c'(w(\xi))\neq 0$.
 
The remaining case to be treated is $|s|\in(\alpha,\beta)$ such that $|s|=c(w(\gamma_{0}))$ and $c'(w(\gamma_{0}))\neq 0$. Using the same notation as in the proof of Lemma~\ref{lem:nvw}, denote by $u_1(t,x)=w_1(x-st)$ and $u_2(t,x)=w_2(x-st)$ the classical solutions to \eqref{eq:nvw} in  $D_{1}$ and $D_{2}$, respectively. Then $w_{1}$ and $w_{2}$ are locally H\"older continuous. Furthermore, we see from \eqref{eq:D1eq} and the corresponding equation for $D_{2}$, that $k_{1}$ and $k_{2}$ are finite and \eqref{eq:Res2} is satisfied for any values of $k_1$ and $k_2$. In particular, the functions $w_1$ and $w_2$ satisfy 
\begin{equation}
\label{eq:lastCase}
	w_{1,\xi}(\xi)=\pm\frac{\sqrt{|k_{1}|}}{\sqrt{|s^{2}-c^{2}(w_{1}(\xi))|}} \quad \text{and} \quad w_{2,\xi}(\xi)=\pm\frac{\sqrt{|k_{2}|}}{\sqrt{|s^{2}-c^{2}(w_{2}(\xi))|}}
\end{equation}
in $D_{1}$ and $D_{2}$, respectively. 

We study the case $w_{1,\xi}(\xi)\rightarrow\pm\infty$ as $\xi\rightarrow\gamma_{0}-$ and $w_{2,\xi}(\xi)\rightarrow\pm\infty$ as $\xi\rightarrow\gamma_{0}+$. It remains to show which solutions, if any, can be glued together.

Let $s>0$. Assume that $s=c(w_{1}(\gamma_{0}))=c(w_{2}(\gamma_{0}))$ and 
$c'(w_{1}(\gamma_{0}))=c'(w_{2}(\gamma_{0}))<0$. Since $c'$, $w_{1}$ and $w_{2}$ are continuous we have $c'(w_{1}(\xi))<0$ for $\xi<\gamma_{0}$ near $\gamma_{0}$ and $c'(w_{2}(\xi))<0$ for $\xi>\gamma_{0}$ near $\gamma_{0}$. We have the following four possibilities:

\textbf{1.} If $w_{1,\xi}(\xi)>0$ for $\xi<\gamma_{0}$ near $\gamma_{0}$ then $c(w_{1}(\xi))>s$ and from \eqref{eq:TWnvw} we get $w_{1,\xi\xi}(\xi)>0$.

\textbf{2.} If $w_{1,\xi}(\xi)<0$ for $\xi<\gamma_{0}$ near $\gamma_{0}$ then $c(w_{1}(\xi))<s$ and \eqref{eq:TWnvw} implies that $w_{1,\xi\xi}(\xi)<0$.

\textbf{3.} If $w_{2,\xi}(\xi)>0$ for $\xi>\gamma_{0}$ near $\gamma_{0}$ then $c(w_{2}(\xi))<s$ and by \eqref{eq:TWnvw} we have $w_{2,\xi\xi}(\xi)<0$.

\textbf{4.} If $w_{2,\xi}(\xi)<0$ for $\xi>\gamma_{0}$ near $\gamma_{0}$ then $c(w_{2}(\xi))>s$ and by \eqref{eq:TWnvw} we have $w_{2,\xi\xi}(\xi)>0$.

We have now $4$ possibilities for gluing waves at $\gamma_0$: 1. and 4., 1. and 3., 2. and 3., and 2. and 4. In all cases the derivatives are unbounded at the gluing point. For instance, combining 1. and 4. results in a wave with a cusp at $\gamma_{0}$. Since the constants $k_{1}$ and $k_{2}$ may differ, $w_{1}$ and $w_{2}$ may have different slope away from $\gamma_{0}$.  

Another possibility, due to \eqref{eq:TWnvw}, is that either $w_{1}$ or $w_{2}$ is constant. We can combine constant solutions with singular waves.
For instance, let $w_{1}(\xi)=w_{2}(\gamma_{0})$ for $\xi\leq\gamma_{0}$, and $w_{2}$ be as in 3. 

We can also combine the wave in 1. with the constant solution where $w_{2}(\xi)=w_{1}(\gamma_{0})$ for $\xi\geq\gamma_{0}$.

A similar analysis can be done in the case $c'(w_{1}(\gamma_{0}))=c'(w_{2}(\gamma_{0}))>0$.

Note that the resulting waves may be unbounded. This is for example the case if $c(u)=\vert s\vert$ for exactly one $u\in \mathbb{R}$. On the other hand, the resulting waves belong to $L^2(D)$ and are locally H{\"o}lder continuous.

Finally we study an example to illustrate how we can glue local waves to get a bounded traveling wave. Let us consider a function $c$ as depicted in Figure~\ref{fig:Figc} and the wave composed of 1. and 3.. For $\xi<\gamma_{0}$ near $\gamma_{0}$, it is given by $w_{1}(\xi)$ which is strictly increasing and convex. For $\xi>\gamma_{0}$ near $\gamma_{0}$, it is given by $w_{2}(\xi)$ which is strictly increasing and concave. In this case we assumed that the function $c$ is strictly decreasing at the point $w_{1}(\gamma_{0})=w_{2}(\gamma_{0})$. Now we assume that $c$ has a local minimum to the right of this point. More precisely, we assume that there exists $E_{1}>\gamma_{0}$ such that $c'(w_{2}(E_{1}))=0$, $c'(w_{2}(\xi))<0$ for all $\gamma_{0}\leq\xi<E_{1}$, and $c'(w_{2}(\xi))>0$ for all $E_{1}<\xi<\xi_{1}$ for $\xi_{1}$ near $E_{1}$, so that $c(w_{2}(\xi))<s$ for all $E_{1}<\xi\leq\xi_{1}$. 

The function $w_{2}(\xi)$ is a strictly increasing classical solution for all $\gamma_{0}<\xi<\xi_{1}$. Furthermore, $w_{2}(\xi)$ has an inflection point at $\xi=E_{1}$ and is concave for $\gamma_{0}\leq\xi<E_{1}$ and convex for $E_{1}<\xi\leq\xi_{1}$.

If $c'(w_{2}(\xi))>0$ for all $\xi_{1}\leq\xi\leq\gamma_{1}$ where $\gamma_{1}$ satisfies $c(w_{2}(\gamma_{1}))=s$, we can continue the wave after $\gamma_1$ either by a singular wave  or by setting $w$ equal to $w(\gamma_{1})$ for $\gamma_{1}<\xi$. The situation is illustrated in Figure \ref{fig:Figw}.

\begin{figure}
	\centering
	\begin{minipage}[t]{0.49\textwidth}
		\includegraphics[width=\textwidth]{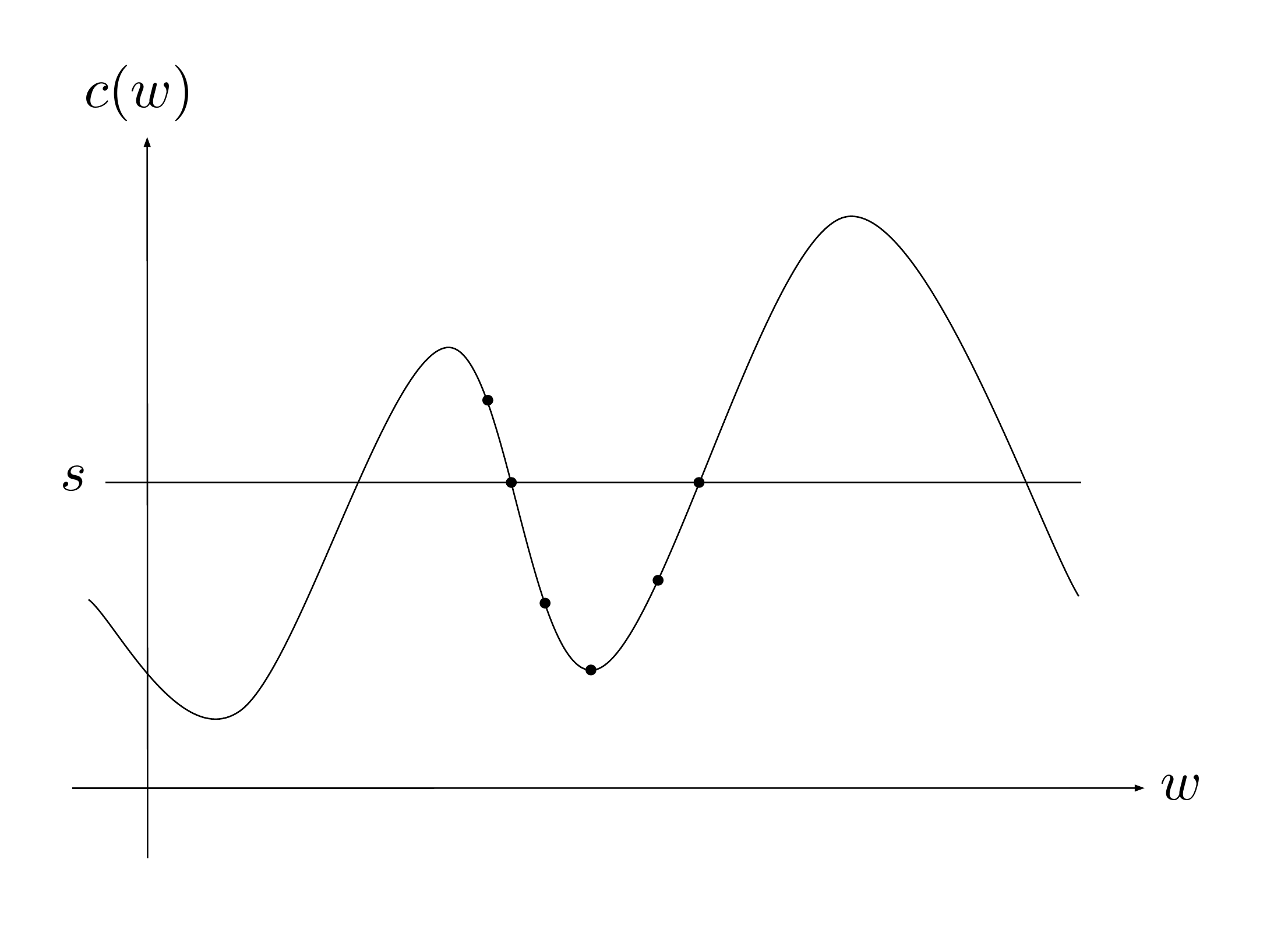}
		\caption{}{The points are from left to right: $w_{1}(\xi)$ for $\xi<\gamma_{0}$, $w_{1}(\gamma_{0})=w_{2}(\gamma_{0})$, $w_{2}(\xi)$ for $\gamma_{0}<\xi<E_{1}$, $w_{2}(E_{1})$, $w_{2}(\xi_{1})$ and $w_{2}(\gamma_{1})$.}
		\label{fig:Figc}
	\end{minipage}
	\hfill
	\begin{minipage}[t]{0.49\textwidth}
		\includegraphics[width=\textwidth]{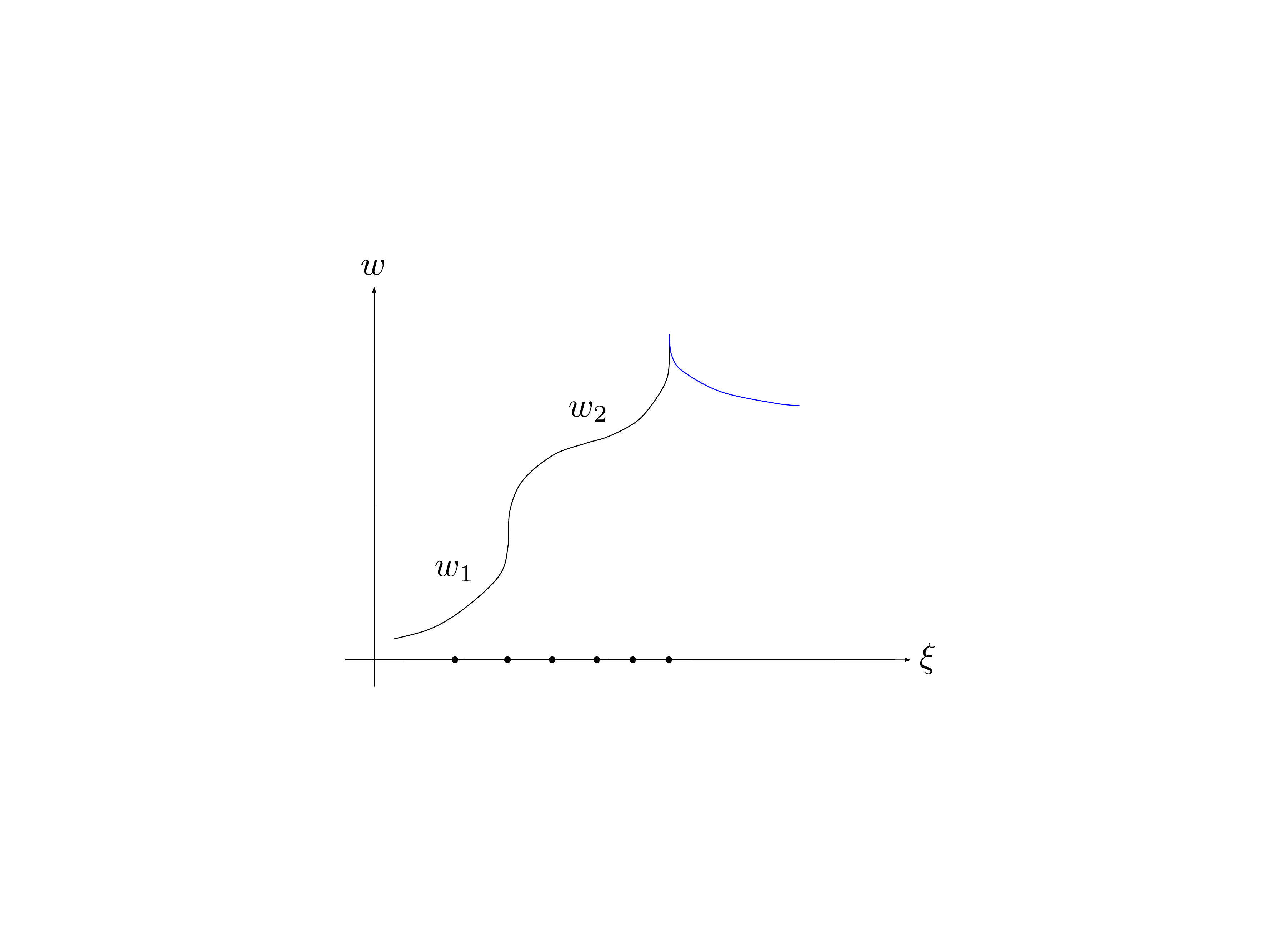}
		\caption{}{The points are from left to right: $\xi<\gamma_{0}$, $\gamma_{0}$, $\gamma_{0}<\xi<E_{1}$, $E_{1}$, $\xi_{1}$ and $\gamma_{1}$. The blue part to the right shows one of the three ways of continuing the wave for $\xi>\gamma_{1}$.}
		\label{fig:Figw}
	\end{minipage}
\end{figure}

Depending on the function $c$, we can continue this gluing procedure to produce a wave $w$ consisting of decreasing, increasing, and constant segments. Note that to get a non-constant bounded traveling wave we have to use increasing or decreasing parts. The derivative of the composite wave belongs to $L^{2}_{\text{loc}}(\mathbb{R})$. If $w_\xi\in L^2(\mathbb{R})$, then $w$ is not only a global traveling wave solution, but also a weak solution. 
\end{proof}

\section{The Camassa--Holm Equation}\label{sec:CH}

Now we study weak traveling wave solutions of the CH equation \eqref{eq:CH}. We insert for a traveling wave solution $u(t,x)=w(x-st)$, and get
\begin{equation}
\label{eq:TWCH}
	-sw_{\xi}+sw_{\xi\xi\xi}+3ww_{\xi}-2w_{\xi}w_{\xi\xi}-ww_{\xi\xi\xi} =0.
\end{equation}

We say that $u$ is a local, classical traveling wave solution of \eqref{eq:CH} if $u(t,x)=w(x-st)$ for some $w$ in $C^{3}(I)$, where $I$ denotes some interval, and satisfies \eqref{eq:TWCH}. 

Integrating \eqref{eq:TWCH} yields
\begin{equation}
\label{eq:CHTWa}
	-sw+sw_{\xi\xi}+\frac{3}{2}w^{2}-\frac{1}{2}w_{\xi}^{2}-ww_{\xi\xi}=a,
\end{equation}
where $a$ is some integration constant. We multiply with $2w_\xi$, integrate once more and get
\begin{equation}
\label{eq:CHTWb}
	-sw^{2}+w^{3}+w_{\xi}^{2}(s-w)=2aw+b,
\end{equation}
where $b$ is some integration constant.

We study if we can glue together local, classical wave solutions like we did in the previous section for the NVW equation. We are interested in the situation where the composite wave has a discontinuous derivative at the gluing points. 

First we derive a weak form of the CH equation. Assume that we have a bounded solution $u\in C^{3}((0,\infty)\times\mathbb{R})$. We multiply \eqref{eq:CH} by a smooth test function $\phi\in C_{c}^{\infty}((0,\infty)\times\mathbb{R})$ and integrate by parts. This yields
\begin{equation}
\label{eq:weakFormCH}
	\int_{0}^{\infty}\int_{-\infty}^{\infty}
	\big(u\phi_{t}-u\phi_{txx}+\frac{3}{2}u^{2}\phi_{x}+uu_{x}\phi_{xx}+\frac{1}{2}u_{x}^{2}\phi_{x}
	\big)\,dx\,dt=0,
\end{equation}
which serves as a basis for defining weak solutions.

A function $u$ satisfying $u(t,\cdot)\in\Linf(\mathbb{R})\cap H^{1}_{\text{loc}}(\mathbb{R})$ for all $t\geq 0$ is said to be a weak solution of \eqref{eq:CH} if \eqref{eq:weakFormCH} holds for all test functions $\phi$ in $C_{c}^{\infty}((0,\infty)\times\mathbb{R})$. In the case of a traveling wave solution, $u(t,x)=w(x-st)$, \eqref{eq:weakFormCH} reads
\begin{equation*}
	\int_{0}^{\infty}\int_{-\infty}^{\infty}
	\big(w\phi_{t}-w\phi_{txx}+\frac{3}{2}w^{2}\phi_{x}+ww_{\xi}\phi_{xx}+\frac{1}{2}w_{\xi}^{2}\phi_{x}\big)\,dx\,dt.
\end{equation*}	

Assume that $u_{t}$ and $u_{x}$ have discontinuities that move along a smooth curve $\Gamma: x=\gamma(t)$, where $\gamma$ is a strictly increasing function, and that $u$ is a local, classical solution of \eqref{eq:CH} on each side of $\gamma(t)$. 

\begin{lemma}
\label{lem:CH}
Given a curve $\Gamma: x=\gamma(t)=st+\gamma_0$, where $\gamma_0$ is a constant, denote by $D$ a neighborhood of $(\bar t, \gamma(\bar t))\in \Gamma$. Furthermore, let $D=D_1\cup \Gamma\vert_D\cup D_2$, where $D_1$ and $D_2$ are the parts of $D$ to the left and to the right of $\Gamma$, respectively, see Figure~\ref{fig:FigD}. Consider two local, classical traveling wave solution $u_1$ and $u_2$ of \eqref{eq:CH} in $D_1$ and $D_2$, respectively. Assume that we glue these waves at $\Gamma$ to obtain a continuous traveling wave $u(t,x)=w(x-st)$ in $D$, which satisfies 
\begin{equation}\label{cond:trw}
	\iint_{D}\big[w\phi_{t}-w\phi_{txx}+\frac{3}{2}w^{2}\phi_{x}+ww_{\xi}\phi_{xx}+\frac{1}{2}w_{\xi}^{2}\phi_{x}\big]\,dx\,dt=0 \quad \text{ for any } \phi\in C_c^\infty(D).
\end{equation}
	
Then we have $a_{1}=a_{2}$, where $a_{1}$ and $a_{2}$ denote the constants in \eqref{eq:CHTWa} corresponding to the local, classical solutions $u_1(t,x)=w_1(x-st)$ and $u_2(t,x)=w_2(x-st)$ in $D_{1}$ and $D_{2}$, respectively.
	
If $w_{\xi}$ and $w_{\xi\xi}$ are bounded on the curve $x=\gamma(t)$, then
\begin{subequations}
\begin{equation}
\label{eq:Res1CH}
	(w(\gamma_{0})-s)(w_{\xi}(\gamma_{0}-)-w_{\xi}(\gamma_{0}+))=0.
\end{equation}

If $w_{\xi}$ and $w_{\xi\xi}$ may be unbounded on the curve $x=\gamma(t)$, then 
\begin{aalign}
\label{eq:Res2CH}
	&\emph{sign}\big(\big((w(\gamma_{0})-s)w_{\xi}(\gamma_{0})\big)-\big)\sqrt{(w^{2}(\gamma_{0})(w(\gamma_{0})-s)-2a_{1}w(\gamma_{0})-b_{1})(w(\gamma_{0})-s)}\\
	&-\emph{sign}\big(\big((w(\gamma_{0})-s)w_{\xi}(\gamma_{0})\big)+\big)\sqrt{(w^{2}(\gamma_{0})(w(\gamma_{0})-s)-2a_{1}w(\gamma_{0})-b_{2})(w(\gamma_{0})-s)}=0,
\end{aalign}
where $b_{1}$ and $b_{2}$ are the constants in \eqref{eq:CHTWb} corresponding to the local, classical solution $u_1(t,x)=w_1(x-st)$ and $u_2(t,x)=w_2(x-st)$ in $D_1$ and $D_2$, respectively.
\end{subequations}
\end{lemma}

\begin{proof}
We use the same notation as in the proof of	Lemma \ref{lem:nvw}. Moreover, many of the calculations are similar to the ones in the previous proof, and we leave out the details. Using that $u$ is a classical solution in $D_{1}^{\varepsilon}$ and integration by parts, leads to
\begin{aalign}
\label{eq:IntPart1}
	&\iint_{D_{1}^{\varepsilon}}
	\big(u\phi_{t}-u\phi_{txx}+\frac{3}{2}u^{2}\phi_{x}+uu_{x}\phi_{xx}+\frac{1}{2}u_{x}^{2}\phi_{x}
	\big)\,dx\,dt\\
	&=\iint_{D_{1}^{\varepsilon}}
	\bigg[\Big(\frac{3}{2}u^{2}\phi-\frac{1}{2}u_{x}^{2}\phi+uu_{x}\phi_{x}\Big)_{x}
	+\Big(u\phi+u_{x}\phi_{x}\Big)_{t}\bigg]\,dx\,dt-\int_{I_{1}^{\varepsilon}}\big(uu_{xx}\phi+u\phi_{tx}+u_{tx}\phi\big)(t,\gamma_{1}^{\varepsilon}(t))\,dt.
\end{aalign}
Now we apply Green's theorem, and use that the integrand is zero everywhere on $\partial D_{1}^{\varepsilon}$ except on the part corresponding to $\gamma_{1}^{\varepsilon}(t)$. We assume that $u(t,x)=w(x-st)$ is a classical traveling wave solution of \eqref{eq:CH}. Then
\begin{aalign}
\label{eq:WeakForm3}
	&\iint_{D_{1}^{\varepsilon}}
	\big(w\phi_{t}-w\phi_{txx}+\frac{3}{2}w^{2}\phi_{x}+ww_{\xi}\phi_{xx}+\frac{1}{2}w_{\xi}^{2}\phi_{x}\big)\,dx\,dt\\
	&=\int_{I_{1}^{\varepsilon}}
	\bigg(\bigg[-ws+\frac{3}{2}w^{2}-\frac{1}{2}w_{\xi}^{2}
	-(w-s)w_{\xi\xi}\bigg](\gamma_{1}^{\varepsilon}(t)-st)\phi(t,\gamma_{1}^{\varepsilon}(t))\\
	&\hspace{36pt}+\big[(w-s)w_{\xi}\big](\gamma_{1}^{\varepsilon}(t)-st)\phi_{x}(t,\gamma_{1}^{\varepsilon}(t))-w(\gamma_{1}^{\varepsilon}(t)-st)\phi_{tx}(t,\gamma_{1}^{\varepsilon}(t))
	\bigg)\,dt.
\end{aalign}
Similarly, we get
\begin{aalign}
\label{eq:WeakForm2ndTermi}
	&\iint_{D_{2}^{\varepsilon}}
	\big(w\phi_{t}-w\phi_{txx}+\frac{3}{2}w^{2}\phi_{x}+ww_{\xi}\phi_{xx}+\frac{1}{2}w_{\xi}^{2}\phi_{x}\big)\,dx\,dt\\
	&=\int_{I_{2}^{\varepsilon}}
	\bigg(-\bigg[-ws+\frac{3}{2}w^{2}-\frac{1}{2}w_{\xi}^{2}
	-(w-s)w_{\xi\xi}\bigg](\gamma_{2}^{\varepsilon}(t)-st)\phi(t,\gamma_{2}^{\varepsilon}(t))\\
	&\hspace{36pt}-\big[(w-s)w_{\xi}\big](\gamma_{2}^{\varepsilon}(t)-st)\phi_{x}(t,\gamma_{2}^{\varepsilon}(t))+w(\gamma_{2}^{\varepsilon}(t)-st)\phi_{tx}(t,\gamma_{2}^{\varepsilon}(t))
	\bigg)\,dt.
\end{aalign}

From \eqref{eq:CHTWa} and \eqref{eq:CHTWb} we have
\begin{equation}
\label{eq:CHTWClassical1}
	-sw+sw_{\xi\xi}+\frac{3}{2}w^{2}-\frac{1}{2}w_{\xi}^{2}-ww_{\xi\xi}=a_{i}
\end{equation}
and
\begin{equation}
\label{eq:CHTWClassical1ii}
	-sw^{2}+w^{3}+w_{\xi}^{2}(s-w)=2a_{i}w+b_{i}
\end{equation}
in $\overline{D_{i}^{\varepsilon}}$, $i=1,2$. Assume that $w_{\xi}$ and $w_{\xi\xi}$ are bounded on the curve $x=\gamma(t)$. We insert \eqref{eq:CHTWClassical1} in \eqref{eq:WeakForm3}, and get since $w$, $\phi$ and the derivatives of $\phi$ are continuous,
\begin{aalign}
	\label{eq:Term1i}
	&\lim_{\varepsilon\rightarrow 0}\iint_{D_{1}^{\varepsilon}}
	\big(w\phi_{t}-w\phi_{txx}+\frac{3}{2}w^{2}\phi_{x}+ww_{\xi}\phi_{xx}+\frac{1}{2}w_{\xi}^{2}\phi_{x}\big)\,dx\,dt\\
	&=\int_{I}
	\bigg(a_{1}\phi(t,\gamma(t))+\big[(w(\gamma_{0})-s)w_{\xi}(\gamma_{0}-)\big]\phi_{x}(t,\gamma(t))-w(\gamma_{0})\phi_{tx}(t,\gamma(t))
	\bigg)\,dt.
\end{aalign}
Combining this with the limit corresponding to \eqref{eq:WeakForm2ndTermi} yields
\begin{aalign}
\label{eq:WeakFormLastStepi}
	&\iint_{D}
	\big(w\phi_{t}-w\phi_{txx}+\frac{3}{2}w^{2}\phi_{x}+ww_{\xi}\phi_{xx}+\frac{1}{2}w_{\xi}^{2}\phi_{x}\big)\,dx\,dt\\
	&=\int_{I}\bigg(
	(a_{1}-a_{2})\phi(t,\gamma(t))+(w(\gamma_{0})-s)(w_{\xi}(\gamma_{0}-)-w_{\xi}(\gamma_{0}+))
	\phi_{x}(t,\gamma(t))\bigg)\,dt,
\end{aalign}
for all $\phi\in C_{c}^{\infty}(D)$. Now we derive conditions that ensure that the integral in \eqref{eq:WeakFormLastStepi} is equal to zero for all test functions $\phi\in C_{c}^{\infty}(D)$. For a positive number $d$ and constants $k$ and $l$ satisfying $k<l$, consider the domain $\tilde{D}\subset D$ bounded by the lines $t=k$, $t=l$ and $x=\gamma(t)\pm d$. We define the test function $\tilde \phi\in C_c^\infty(D)$ by
\begin{equation}
\label{eq:testFunc}
	\tilde{\phi}(t,x)= \exp\left\{-\frac{1}{d-(x-\gamma(t))^{2}}-\frac{1}{\frac{1}{4}(l-k)^{2}-(t-\frac{k+l}{2})^{2}}\right\},
\end{equation} 
which is positive and smooth in $\tilde{D}$ and equals zero on the boundary $\partial \tilde{D}$. In particular, $\tilde{\phi}_{x}(t,\gamma(t))=0$ for all $t\in[k,l]$. From \eqref{cond:trw} and \eqref{eq:WeakFormLastStepi}, we then get
\begin{equation*}
	\iint_{\tilde{D}}
	\big(w\tilde{\phi}_{t}-w\tilde{\phi}_{txx}+\frac{3}{2}w^{2}\tilde{\phi}_{x}+ww_{\xi}\tilde{\phi}_{xx}+\frac{1}{2}w_{\xi}^{2}\tilde{\phi}_{x}\big)\,dx\,dt=\int_{a}^{b}\big[a_{1}-a_{2}\big]\tilde{\phi}(t,\gamma(t))\,dt=0,
\end{equation*} 
and since $\tilde{\phi}$ is positive in $\tilde{D}$ this implies that $a_{1}=a_{2}$. Furthermore, since \eqref{eq:WeakFormLastStepi} should be equal to zero for all test functions $\phi$, we must have
\begin{aalign}
	\label{eq:WeakFormPhixi}
	&\iint_{D}
	\big(w\phi_{t}-w\phi_{txx}+\frac{3}{2}w^{2}\phi_{x}+ww_{\xi}\phi_{xx}+\frac{1}{2}w_{\xi}^{2}\phi_{x}\big)\,dx\,dt\\
	&=\int_{I}(w(\gamma_{0})-s)(w_{\xi}(\gamma_{0}-)-w_{\xi}(\gamma_{0}+))
	\phi_{x}(t,\gamma(t))\,dt=0
\end{aalign} 
for all $\phi\in C_{c}^{\infty}(D)$, which implies
$(w(\gamma_{0})-s)(w_{\xi}(\gamma_{0}-)-w_{\xi}(\gamma_{0}+))=0$. This proves \eqref{eq:Res1CH}.

Now assume that $w_{\xi}$ and $w_{\xi\xi}$ may be unbounded on the curve $x=\gamma(t)$. Recall that $w$ is a classical solution in $D_{i}^{\varepsilon}$, $i=1,2$. 

Using \eqref{eq:CHTWClassical1ii} we write
\begin{equation*}
	(w-s)w_{\xi}=\text{sign}\big((w-s)w_{\xi}\big)\sqrt{(w^{2}(w-s)-2a_{1}w-b_{1})(w-s)}
\end{equation*}
in $\overline{D_{1}^{\varepsilon}}$, which together with \eqref{eq:WeakForm3} and \eqref{eq:CHTWClassical1} yields
\begin{align*}
	&\lim_{\varepsilon\rightarrow 0}\iint_{D_{1}^{\varepsilon}}
	\big(w\phi_{t}-w\phi_{txx}+\frac{3}{2}w^{2}\phi_{x}+ww_{\xi}\phi_{xx}+\frac{1}{2}w_{\xi}^{2}\phi_{x}\big)\,dx\,dt\\
	&=\int_{I}\bigg(a_{1}\phi(t,\gamma(t))-w(\gamma_{0})\phi_{tx}(t,\gamma(t))+\text{sign}\big(\big((w(\gamma_{0})-s)w_{\xi}(\gamma_{0})\big)-\big)\\
	&\hspace{40pt}\times\sqrt{(w^{2}(\gamma_{0})(w(\gamma_{0})-s)-2a_{1}w(\gamma_{0})-b_{1})(w(\gamma_{0})-s)}\phi_{x}(t,\gamma(t))\bigg)\,dt.
\end{align*}
Combining this with the limit corresponding to \eqref{eq:WeakForm2ndTermi} yields
\begin{align*}
	&\iint_{D}\big(w\phi_{t}-w\phi_{txx}+\frac{3}{2}w^{2}\phi_{x}+ww_{\xi}\phi_{xx}+\frac{1}{2}w_{\xi}^{2}\phi_{x}\big)\,dx\,dt=\int_{I}\bigg((a_{1}-a_{2})\phi(t,\gamma(t))\\
	&\quad+\Big(\text{sign}\big(\big((w(\gamma_{0})-s)w_{\xi}(\gamma_{0})\big)-\big)\sqrt{(w^{2}(\gamma_{0})(w(\gamma_{0})-s)-2a_{1}w(\gamma_{0})-b_{1})(w(\gamma_{0})-s)}\\
	&\quad-\text{sign}\big(\big((w(\gamma_{0})-s)w_{\xi}(\gamma_{0})\big)+\big)\sqrt{(w^{2}(\gamma_{0})(w(\gamma_{0})-s)-2a_{2}w(\gamma_{0})-b_{2})(w(\gamma_{0})-s)}\Big)\phi_{x}(t,\gamma(t))
	\bigg)\,dt
\end{align*}
for all $\phi\in C_{c}^{\infty}(D)$.

As before, we choose the test function $\tilde{\phi}\in C_c^\infty(D)$ given by \eqref{eq:testFunc} to obtain $a_{1}=a_{2}$.  Thus, the first term in the above integral drops out, and for the remaining integral to be equal to zero for all test functions, we must have
\begin{align*}
	&\text{sign}\big(\big((w(\gamma_{0})-s)w_{\xi}(\gamma_{0})\big)-\big)\sqrt{(w^{2}(\gamma_{0})(w(\gamma_{0})-s)-2a_{1}w(\gamma_{0})-b_{1})(w(\gamma_{0})-s)}\\
	&-\text{sign}\big(\big((w(\gamma_{0})-s)w_{\xi}(\gamma_{0})\big)+\big)\sqrt{(w^{2}(\gamma_{0})(w(\gamma_{0})-s)-2a_{1}w(\gamma_{0})-b_{2})(w(\gamma_{0})-s)}=0.
\end{align*}
This concludes the proof of \eqref{eq:Res2CH}.
\end{proof}

We apply Lemma \ref{lem:CH} to study which local, classical traveling waves we can glue together.

\textbf{Bounded derivatives}

\textbf{Case 1.} Consider $s\neq w(\gamma_{0})$. For \eqref{eq:Res1CH} to be satisfied we must have $w_{\xi}(\gamma_{0}-)=w_{\xi}(\gamma_{0}+)$, i.e., the derivative is continuous at $\gamma_{0}$. From \eqref{eq:CHTWClassical1ii} we then get that $b_{1}=b_{2}$. If $w_\xi(\gamma_{0}-)=w_\xi(\gamma_{0}+)\not =0$, then $w$ coincides with the local solution in $D$ of the differential equation \eqref{eq:CHTWClassical1ii}. If on the other hand $w_\xi(\gamma_{0}-)=w_\xi(\gamma_{0}+)=0$, we have two possibilities: Either $w$  coincides with the local solution in $D$ of the differential equation \eqref{eq:CHTWClassical1ii} or $w_\xi$ changes sign at $\gamma_0$, so that $w$ has a local maximum or minimum at $\gamma_0$. The latter case occurs when constructing periodic solutions.

\textbf{Case 2.} If $s=w(\gamma_{0})$, \eqref{eq:Res1CH} is satisfied. Since $u(t,\gamma(t))=w(\gamma(t)-st)=w(\gamma_{0})$ we get $u(t,\gamma(t))=s$ for all $t\geq 0$. We denote the solution in $D_{1}$ and $D_{2}$ by $w_{1}$ and $w_{2}$, respectively. Then $w_{1}(\gamma_{0})=w_{2}(\gamma_{0})=s$. We let $\xi$ tend to $\gamma_{0}$ in \eqref{eq:CHTWClassical1} and obtain $a_{1}=\frac{1}{2}w_{1}^{2}(\gamma_{0})-\frac{1}{2}w_{1,\xi}^{2}(\gamma_{0})$ and $a_{2}=\frac{1}{2}w_{2}^{2}(\gamma_{0})-\frac{1}{2}w_{2,\xi}^{2}(\gamma_{0})$. Since $a_{1}=a_{2}$ this implies that $w_{1,\xi}^{2}(\gamma_{0})=w_{2,\xi}^{2}(\gamma_{0})$. 

If $w_{1,\xi}(\gamma_{0})=w_{2,\xi}(\gamma_{0})$, one has $b_1=b_2$ and as in Case 1 $w$ coincides with the local solution in $D$ of the differential equation \eqref{eq:CHTWClassical1ii}.

If $w_{1,\xi}(\gamma_{0})=-w_{2,\xi}(\gamma_{0})$, we get from \eqref{eq:CHTWClassical1ii}, that $2a_{1}s+b_{1}=0$ and $2a_{1}s+b_{2}=0$, which implies $b_{1}=b_{2}$. Furthermore, \eqref{eq:CHTWClassical1ii} takes the form
\begin{equation*}
\big(w_{1,\xi}^{2}(\xi)-w_{1}^{2}(\xi)+2a_{1}\big)\big(s-w_{1}(\xi)\big)=0
 \quad \text{and} \quad 
\big(w_{2,\xi}^{2}(\xi)-w_{2}^{2}(\xi)+2a_{1}\big)\big(s-w_{2}(\xi)\big)=0
\end{equation*}
in $\overline{D_{1}^{\varepsilon}}$ and $\overline{D_{2}^{\varepsilon}}$, respectively. Assuming that $w_{1}$ and $w_{2}$ are not constant and equal to $s$ near $\gamma_{0}$, we have
\begin{equation}
\label{eq:EQN1}
w_{1,\xi}^{2}-w_{1}^{2}+2a_{1}=0 \quad \text{and} \quad w_{2,\xi}^{2}-w_{2}^{2}+2a_{1}=0.
\end{equation}
For this to be well-defined we require $w_{1}^{2}-2a_{1}\geq 0$ and $w_{2}^{2}-2a_{1}\geq 0$, and in particular,
\begin{equation}\label{eq:condpeak}
w_{1,\xi}=\pm\sqrt{w_{1}^{2}-2a_{1}} \quad \text{and} \quad w_{2,\xi}=\pm\sqrt{w_{2}^{2}-2a_{1}}.
\end{equation}
Note that $w_1$ and $w_2$ can only change from increasing to decreasing or the other way round if $w_1^2=2a_1$. We differentiate \eqref{eq:EQN1} and get
\begin{equation*}
w_{1,\xi}(w_{1,\xi\xi}-w_{1})=0 \quad \text{and} \quad w_{2,\xi}(w_{2,\xi\xi}-w_{2})=0,
\end{equation*}
and since the solutions are not constant we have
\begin{equation}
\label{eq:EQN2}
w_{1,\xi\xi}=w_{1} \quad \text{and} \quad w_{2,\xi\xi}=w_{2}.
\end{equation}
Letting $\xi$ tend to $\gamma_{0}$ in \eqref{eq:EQN2} yields
$w_{1,\xi\xi}(\gamma_{0})=w_{1}(\gamma_{0})$ and $w_{2,\xi\xi}(\gamma_{0})=w_{2}(\gamma_{0})$, so that if $s$ is positive then $w_{1}$ and $w_{2}$ are convex. Since the functions are not constant and $w_{1,\xi}(\gamma_{0})=-w_{2,\xi}(\gamma_{0})$, this implies that $w_{1}$ is increasing and $w_{2}$ is decreasing near $\gamma_{0}$. Otherwise the resulting function $w$ would be globally unbounded. Thus, the maximum value of $w_{1}$ and $w_{2}$ near $\gamma_{0}$ is attained at $\gamma_{0}$ where $w_{1}(\gamma_{0})=w_{2}(\gamma_{0})=s$.

If $s$ is negative, $w_{1}$ and $w_{2}$ are concave, and $w_{1}$ is decreasing and $w_{2}$ is increasing. Otherwise the resulting function $w$ would be globally unbounded. The minimum value of $w_{1}$ and $w_{2}$ near $\gamma_{0}$ is attained at $\gamma_{0}$ where $w_{1}(\gamma_{0})=w_{2}(\gamma_{0})=s$.

\begin{example}
	Let $\gamma_{0}=0$. Then
	\begin{equation*}
	w_{1}(\xi)=c_{1}e^{\xi}+c_{2}e^{-\xi} \quad \text{and} \quad w_{2}(\xi)=c_{2}e^{\xi}+c_{1}e^{-\xi},
	\end{equation*}
	where $c_{1}$ and $c_{2}$ are constants satisfying $\frac{c_{1}}{c_{2}}>e^{2}$, solve the differential equations \eqref{eq:EQN2} in $[-1,0]$ and $[0,1]$, respectively. Observe that $w_{1}$ is increasing in $[-1,0]$, $w_{2}$ is decreasing in $[0,1]$, $w_{1}(0)=w_{2}(0)$ and $w_{1,\xi}(0)=-w_{2,\xi}(0)$.
\end{example}

\begin{remark}
Note that the so-called multipeakon solutions are of the form $u(t,x)=\sum_{i=1}^n p_i(t)e^{-\vert x-q_i(t)\vert}$. Thus if one only glues together local, traveling waves, which have bounded derivatives, one ends up with a multipeakon solution due to \eqref{eq:EQN2}.
\end{remark}

\textbf{Unbounded derivatives}

Since we have from \eqref{eq:CHTWClassical1ii}, $w_{\xi}^{2}=w^{2}-\frac{2a_{1}w+b_{i}}{w-s}$ in $\overline{D_{i}^{\varepsilon}}$, it follows that $s=w(\gamma_0)$ at the possible glueing point. Furthermore, due to \eqref{eq:Res2CH} the constants $b_{1}$ and $b_{2}$ do not have to be identical. 

Note that \eqref{eq:Res2CH} implies that it is possible to glue together both constant and non-constant local, classical solutions as long as $s=w(\gamma_0)$. This means in particular that one can insert constant parts by gluing. 

\vspace{5pt}

In \cite{Len}, Lenells presents a complete classification of weak, bounded traveling waves for the CH equation. He shows that there exists a wide range of waves, such as smooth waves, but also peakons, cuspons, stumpons, and composite waves which might have singularities. 

Lenells proves that two traveling waves $w_{1}$ and $w_{2}$ can only be glued together at a point $\xi$ if the wave height equals the wave speed, i.e., $w_{1}(\xi)=w_{2}(\xi)=s$, and if the constants $a_{1}$ and $a_{2}$ from \eqref{eq:CHTWClassical1ii} are equal. We remark that the constants $a$ and $c$ in \cite{Len} corresponds to $2a$ and $s$ here, respectively, and that we assume $k=0$. 

Our main objective was to recover these conditions by using the method presented above. Showing other important features of traveling wave solutions of the CH equation requires the machinery used by Lenells, which we outline next. For a detailed description we refer to \cite{Len}. A key property is that the maximum value of the wave equals $s$ for $s>0$ and the minimum value equals $s$ for $s<0$.

In particular, we highlight the role the constant $b_{i}$ plays in obtaining a bounded wave. Assume that we are in our usual setting where we have classical solutions $w_{1}$ and $w_{2}$ in $D_{1}^{\varepsilon}$ and $D_{2}^{\varepsilon}$, respectively. We want to glue these waves together. Thus, we must have $w_{1}(\gamma_{0})=w_{2}(\gamma_{0})=s$ and $a_{1}=a_{2}$. Hence, by introducing $f(w)=-w^{3}+sw^{2}+2a_{1}w$, we can write \eqref{eq:CHTWClassical1ii} as 
\begin{equation}
\label{eq:wxi}
	w_{\xi}^{2}(s-w)=f(w)+b_{i}=g_{i}(w) 
\end{equation}
in $\overline{D_{i}^{\varepsilon}}$, $i=1,2$. Note that $g_{1}'(w)=g_{2}'(w)=f'(w)=-3w^{2}+2sw+2a_{1}$. In what follows we assume that $s>0$.

If $s^{2}+6a_{1}\leq 0$, then $g_{1}'(w)\leq -3w^{2}+2sw-\frac{s^{2}}{3}=-3\big(w-\frac{s}{3}\big)^{2}$, which is strictly negative provided that $w$ is not identically equal to $\frac{s}{3}$. This means that $g_{1}(w)$ is strictly decreasing and $f(w)+b_{1}$ has exactly one zero. Assume that $b_{1}$ is such that $f(s)+b_{1}<0$. By continuity we have $f(w)+b_{1}<0$ for $w$ near $s$. Then \eqref{eq:wxi} implies that $w>s$. Since $f(w)+b_{1}<0$ for all $w>s$, \eqref{eq:wxi} shows that $w_{\xi}\neq 0$ for all $w>s$. Thus, $w$ is strictly monotone and unbounded. Next, let us set $b_{1}$ larger so that $f(s)+b_{1}>0$. Then $f(w)+b_{1}>0$ for $w$ near $s$ and from \eqref{eq:wxi} we get $w<s$. We have $f(w)+b_{1}>0$ for all $w<s$, so \eqref{eq:wxi} implies that $w_{\xi}\neq 0$ for all $w<s$. Hence, $w$ is strictly monotone and unbounded. The situation $f(s)+b_{1}=0$ can be treated similarly, showing that there are no bounded solutions. Thus, if $s^{2}+6a_{1}\leq 0$, $f(w)+b_{1}$ has one zero and there exist no bounded solutions to \eqref{eq:wxi}.

If $s^{2}+6a_{1}>0$, then $g_1(w)=f(w)+b_1$ has at least one zero, but the number of zeros is dependent on the choice of $b_{1}$. To be more precise the function $g_{1}(w)$ has a local minimum and maximum at
$w_{\text{min}}=\frac{s-\sqrt{s^{2}+6a_{1}}}{3}$ and $w_{\text{max}}=\frac{s+\sqrt{s^{2}+6a_{1}}}{3}$, respectively. It is strictly decreasing for $w<w_{\text{min}}$ and $w>w_{\text{max}}$.

If $g_1$ has only one zero, we can show as before (i.e. in the case $s^2+6a_1\leq 0$), that there only exist unbounded solutions.

Consider the case where $g_{1}$ has three zeros. Moreover, we consider $s$ between two of the zeros of $g_{1}$, since any other case yields unbounded solutions. 

First we treat the case where $g_{1}$ has a double zero and a simple zero. The double zero is either the local minimum or maximum of $g_{1}$. We only consider the case where the double zero is the local minimum of $g_{1}$, see Figure~\ref{DZero2}, since the other one follows the same lines. Denoting the simple zero by $\eta$, we write
\begin{figure}
	\centering
	\begin{minipage}[t]{0.49\textwidth}
		\includegraphics[width=\textwidth]{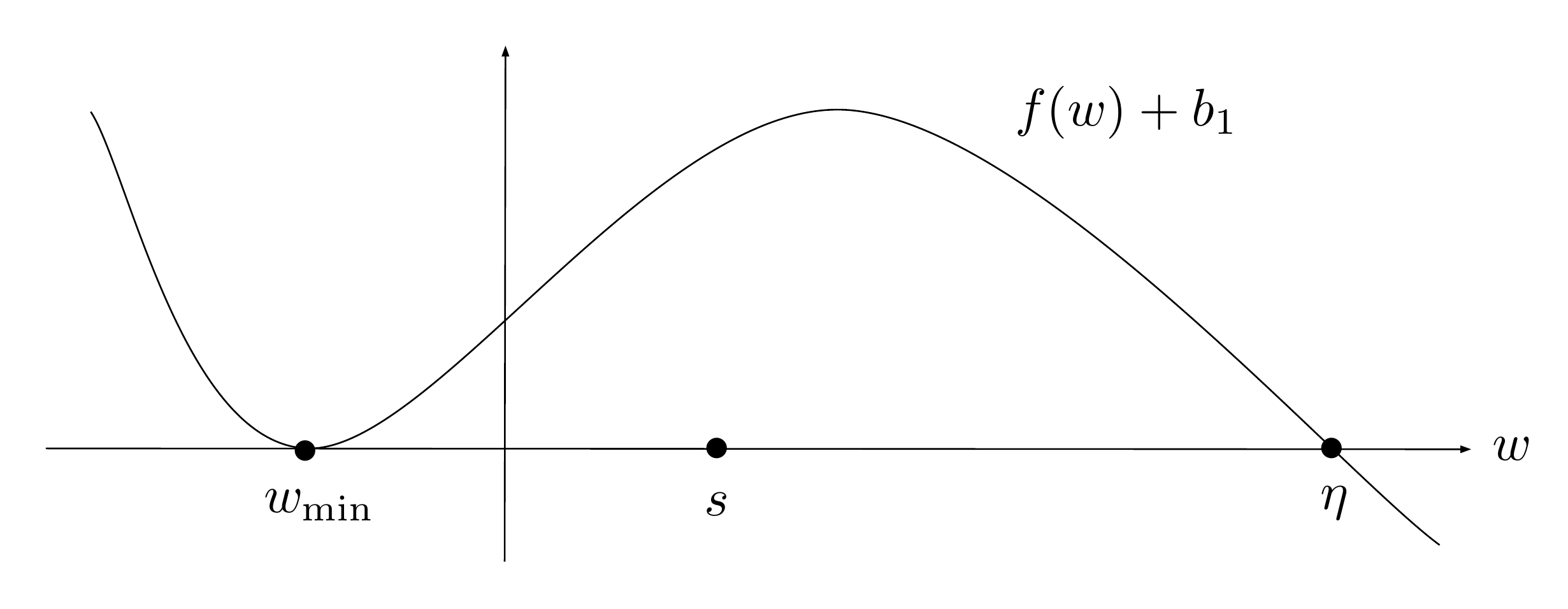}
		\caption{}{Sketch of the function $g_1$ with a double zero at $w=w_{\min}$ and a simple zero at $w=\eta$. Furthermore, $w_{\min}<s<\eta$.}
		\label{DZero2}
	\end{minipage}
	\hfill
	\begin{minipage}[t]{0.49\textwidth}
		\includegraphics[width=\textwidth]{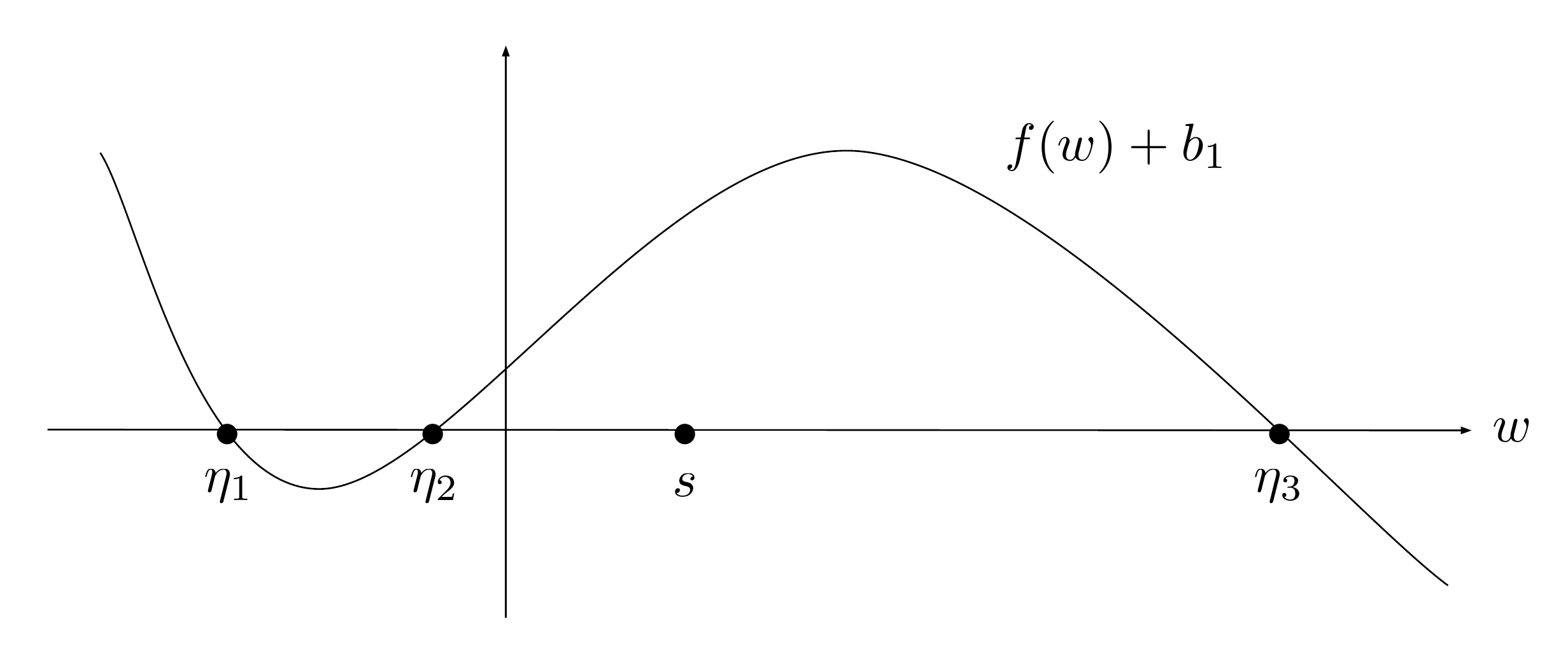}
		\caption{}{Sketch of the function $g_1$ with three simple zeros $\eta_1<\eta_2<\eta_3$. Furthermore, $\eta_2<s<\eta_3$.}
		\label{SZero2}
	\end{minipage}
\end{figure}
\begin{equation}\label{Lcases}
	g_{1}(w)=-(w-w_{\text{min}})^{2}(w-\eta).
\end{equation}
Expanding this expression and comparing it with $g_{1}(w)=f(w)+b_{1}$ we get the relations $\eta+2w_{\text{min}}=s$, $-2\eta w_{\text{min}}-w_{\text{min}}^{2}=2a_{1}$, and $\eta w_{\text{min}}^{2}=b_{1}$.

By assumption $w_{\text{min}}<s<\eta$. Then $g_{1}(s)>0$, so that $g_{1}(w)>0$ for $w$ near $s$. By \eqref{eq:wxi}, $w<s$. Note that $w_{\xi}=0$ for the value $w=w_{\text{min}}$. We have
\begin{equation}
\label{eq:cuspDecay}
	w_{\xi}=\pm(w-w_{\text{min}})\sqrt{\frac{\eta-w}{s-w}}
\end{equation}
and $w$ exponentially decays to $w_{\text{min}}$ as $\xi\rightarrow\pm\infty$. With the notation above, we see that one possibility is to choose $w_1$ and $w_2$ to be solutions to \eqref{eq:cuspDecay}, which yields the cuspon with decay. In particular, $w_{1}$ is the strictly increasing part, and $w_{2}$ the strictly decreasing part. Note that the derivatives are unbounded at $w_{1}=w_{2}=s$. 

Let us investigate if we can glue waves $w_i$, $i=1,2$ given by \eqref{eq:cuspDecay} to constant solutions. Consider $w_1$ given by \eqref{eq:cuspDecay}. From \eqref{eq:wxi}, we have
$w_{2,\xi}^{2}=\frac{g_{2}(w_{2})}{s-w_{2}}$. We are looking for solutions satisfying $w_{2,\xi}=0$. Since at the gluing point we have $w_{1}=w_{2}=s$, we require that $g_{2}(w)=(d-w)(s-w)^2$ for some constant $d$. Comparing the coefficients, yields the relations $d=-s$, $2a_1=s^2$, and $b_2=-s^3$. Hence, if $s^{2}=2a_{1}$ we can glue $w_{i}$ as given by \eqref{eq:cuspDecay} with the constant solution $w_{i\pm1}=s$, which are the building blocks for so-called stumpons.

\begin{remark}
Note that the condition $s^2=2a_1$ is related to \eqref{eq:condpeak}, which describes all local, classical traveling waves that have a bounded derivative at points where $w=s$. In particular, \eqref{eq:condpeak} implies that peakons can only turn up in bounded, composite waves such that $s^2\geq 2a_1$ and the case $s^2=2a_1$ corresponds to the constant solution.
\end{remark} 

Let us turn back to \eqref{Lcases}. If $s=\eta$, then
\begin{equation}
\label{eq:peakDecay}
	w_{\xi}=\pm(w-w_{\text{min}}).
\end{equation} 
In particular, $w$ is monotone and decays to $w_{\text{min}}$ as $\xi\rightarrow\pm\infty$. Choosing $w_1$ and $w_2$ to be solutions to \eqref{eq:peakDecay} yields the peakon with decay. In particular, $w_{1}$ is the strictly increasing part and $w_{2}$ the strictly decreasing one.

Let us see if we can glue waves given by \eqref{eq:peakDecay} to constant solutions. Let $w_{1}$ be the strictly increasing function given by \eqref{eq:peakDecay}. The graph of the function $g_{2}(w)=f(w)+b_{2}$ is equal to the one for $g_{1}$ up to a vertical shift, i.e., there is a constant $\alpha$ such that $g_{2}(w)=-(w-w_{\text{min}})^{2}(w-s)+\alpha$. From \eqref{eq:wxi}, we get $w_{2,\xi}^{2}=(w_{2}-w_{\text{min}})^{2}+\frac{\alpha}{s-w}$. Observe that the only choice of the constant $\alpha$ which gives a solution with bounded derivative is $\alpha=0$. Then we get
$w_{2,\xi}^{2}=(w_{2}-w_{\text{min}})^{2}$, and the only possibility for $w_{2,\xi}=0$ is if $w_{2}=w_{\text{min}}$. But then $w_{2}$ can not be glued to $w_{1}$, as $w_{\text{min}}\neq s$. We conclude that waves given by \eqref{eq:peakDecay} cannot be glued to constant solutions.

Next we treat the case where $g_{1}$ has three simple zeros $\eta_{1}<\eta_{2}<\eta_{3}$, i.e., $g_{1}(w)=-(w-\eta_{1})(w-\eta_{2})(w-\eta_{3})$, see Figure~\ref{SZero2}.

Let $s$ be such that $\eta_{2}<s<\eta_{3}$. Then $g_{1}(s)>0$, so that $g_{1}(w)>0$ for $w$ near $s$. By \eqref{eq:wxi}, $w<s$. Observe that $w_{\xi}=0$ at $w=\eta_{2}$. We have
\begin{equation}
\label{eq:cuspPer}
	w_{\xi}=\pm\sqrt{w-\eta_{2}}\,h(w),
\end{equation}
where $h(w)=\sqrt{-\frac{(w-\eta_{1})(w-\eta_{3})}{s-w}}>0$ for all $\eta_{2}<w<s$, so that $w$ attains the value $\eta_{2}$ at some finite point $\bar \xi$. Note that the solution to \eqref{eq:cuspPer} is not unique. Thus, $w\in C^1$ can be defined in such a way that $w$ attains its minimum at $\bar\xi$, is strictly decreasing to the left of $\bar \xi$, and strictly increasing to the right of $\bar \xi$. Gluing countably many of these waves together yields a periodic cuspon. 

If $s=\eta_{3}$, then
\begin{equation}
\label{eq:peakPer}
	w_{\xi}=\pm\sqrt{w-\eta_{1}}\sqrt{w-\eta_{2}},
\end{equation}
whose solutions, following the same lines as above, serves as building blocks for a periodic peakon.

In a similar way as above, we can study if the waves given by \eqref{eq:cuspPer} and \eqref{eq:peakPer} can be glued to constant solutions. We find that we can only glue waves given by \eqref{eq:cuspPer} to constant solutions which are equal to $s$. Then we obtain stumpons, which consist of monotone segments glued at points where the derivative is unbounded to piecewise constants parts.

\end{document}